\documentclass[english,11pt]{article}

\usepackage[german,french,english]{babel}
\usepackage[cp850]{inputenc}
\usepackage{latexsym,graphicx, fancybox}
\usepackage{graphicx}
\usepackage{amssymb, amsmath, amsfonts}
\usepackage{color}
\usepackage{mathabx}
\usepackage{latexsym}

\newcommand{\vertiii}[1]{{\left\vert\kern-0.25ex\left\vert\kern-0.25ex\left\vert #1 
    \right\vert\kern-0.25ex\right\vert\kern-0.25ex\right\vert}}

\newcommand{\ve}{\varepsilon}

\def\R{\mathbb{R}}

\def\P{\mathbb{P}}
\def\E{\mathbb{E}}

\def\N{{\rm I\hspace{-0.50ex}N} } 
\def\N{\mathbb{N}}

\newtheorem{lemma}{Lemma}[section]
\newtheorem{theorem}{Theorem}[section]
\newtheorem{corollary}{Corollary}[section]
\newtheorem{definition}{Definition}[section]
\newtheorem{proposition}{Proposition}[section]

\begin{document}
\title{\bf  Zador  Theorem for optimal quantization with respect to Bregman divergences}
\author{ 
{\sc Guillaume Boutoille} \thanks{Laboratoire de Probabilit\'es, Statistique et Mod\'elisation, Sorbonne Universit\'e, case 158, 4, pl. Jussieu, F-75252 Paris Cedex 5, France  \& Fotonower, 30 rue Charlot, F-75003 Paris. E-mail: {\tt  guillaume.boutoille@orange.fr}}
\and   
{\sc  Gilles Pag\`es} \thanks{Laboratoire de Probabilit\'es, Statistique et Mod\'elisation,  Sorbonne Universit\'e. E-mail: {\tt  gilles.pages@sorbonne-universit\'e.fr}}~\thanks{This research
benefited from the support of the ``Chaire Risques Financiers'', Fondation du Risque.} }

	\maketitle
	
	\begin{abstract} 
	We establish a Zador like theorem for $L^r$-optimal vector quantization when  the similarity measure is  a twice differentiable Bregman divergence of a strictly convex function. On our way we also prove a similar result when the Bregman divergence is replaced by a continuous matrix-valued vector field having values in the set of positive definite matrices. We adopt the strategy of the first fully rigorous proof of the original Zador' theorem (when the similarity measure is the power of  a norm). We have to overcome several difficulties which are specific to this framework especially concerning the so-called firewall lemma.
		\end{abstract}
	
	\medskip
\newtheorem{step}{Step}
\section{Introduction}
In computer vision, labeling represents a major cost that we aim to reduce as much as possible. Clustering algorithms are useful tools that aim to partition similar data into clusters in order to organize data set. Browsing these clusters allows a better visualisation of the data set and easier labeling. 

Clustering is a fundamental ``unsupervised" learning procedure that has been widely studied across many  disciplines. Most of the clustering methods consist in partitioning similar data into clusters with cluster representatives, also known as "codebook", such that codebook minimize loss function (quantization error). A widely used and studied clustering algorithm is the Euclidean $k$-means algorithm (\cite{macqueen67,jain88}). In~\cite{GrafL2000}, the authors  have shown that a discrete set of centers converges to a measure closely related to the underlying probability distribution.
The asymptotic analysis is important since we want to study a large set of centers that approximates the probability distribution of the data.
In this paper data are projection images onto hyperplanes generated by a neural network model with a large number of data. 
With some complex data, we might use different similarity measures to  partition i.e. classify the data set.

Bregman divergences are broad classes of  similarity measures indexed by  strictly convex functions. A lot of well-known   similarity measures like Euclidean, Mahalanobis, Kullback-Leibler and SoftPlus (a.k.a SoftAbs) divergences are particular cases of Bregman divergences. In $\R^d$, the Bregman divergence $\phi_{_F}$ induced by a strictly convex function $F:U\to \R^d$ is defined as
$$
\forall\, a,\, b\!\in U, \qquad \phi_{_F}(a,b) = F(a) - F(b) - \langle \nabla F(b)\,|\, a-b\rangle,
$$
where $\langle\cdot\,|\,\cdot \rangle$ is the inner product and $\nabla F(b)$ is the gradient of $F$ at $b$. Note that if $F(x)=|x|^2_{_S}:= x^*Sx$ is an Euclidean norm, $S\!\in{\cal S}^{++}(d,\R)$ (i.e. $S$  is symmetric and positive definite), then $\phi_{_F}(a,b)= (a-b)^*S(a-b)= |a-b|^2_{_S}$. 
 In~\cite{Banerjeeetal2005}, the authors have shown a close relation between regular Bregman divergences and log--likelihood of exponential families and   generalized the $k$-means clustering algorithm using  these similarity measures. This work gave rise  to a new field of research  about clustering and quantization where the loss function is a  Bregman divergence  in finite and infinite dimensions (see e.g.~\cite{Fischer2010} or~\cite{LiuBelNIPS2016}).
 
 The aim of this paper is to establish in a mathematically rigorous way the counterpart of Zador's Theorem in this framework i.e. a sharp rate of decay of the $(r, \Phi_{_F})$-mean  quantization error at rate $n^{-1/d}$ as the level $n$ of quantization goes to infinity. This the object of Theorem~\ref{thm:ZadorBregman} and Section~\ref{sec:proofZadorBregman}. Formally speaking, the sharp asymptotic  rate in a Bregman divergence setting for Zador's Theorem makes appear the Hesssian of $\Phi_{_F}$ in the  limiting constant: namely  $\big\|\det(\nabla^2 F)^{\frac{r}{2d}}\cdot h \big\|_{\frac{d}{d+r}}$ instead of $\|h\|_{\frac{d}{d+r}}$  where $h$ denotes the  density of the absolutely continuous component of the distribution $P$ to be quantized.  Our approach differs from that developed in~\cite{LiuBelNIPS2016} not only in terms of method of proof, but also as concerns assumptions.   These poins are discussed after the statement of theorem~\ref{thm:ZadorBregman}. We extend these result to fields of positive definite symmetric matrices in Section~\ref{sec:Zadorfieldmatrix}.

For recent developments on the original Zador  Theorem  i.e. when the similarity measure  is a power of a norm (see Theorem~\ref{thm:reg_zador}  in the next section), we refer to~\cite{LuPag23} which includes some recent improvement on the moment  assumption when the distribution $P$ is radial. 

Such  result   can be considered as intuitive if one thinks to the so-called Mahalanobis divergence $\phi(\xi,x)= (\xi-x)^*S(\xi-x)$ where $S\!\in {\cal S}^{++}(d, \R)$ (positive definite matrix, see Section~\ref{subsec:Mahalanobis} further on) which is in fact contained in the classical family of the norms as loss functions. A rather general result is   stated in an informal way \color{black}in the Neurips communication~\cite{LiuBelNIPS2016} and its supplementary material.  Their approach relies to a large extent on the fact that $\xi\simeq x$ in $\R^d$, then $\phi_{_F}(\xi,x) \simeq (\xi-x)\nabla^2F(X) (\xi-x)^{\top}$.  Ours directly considers the Bregman divergence which will lead to slightly different assumptions  on $F$. Some comments are provided as a remark right below the theorem. Our proof is   in line with that  developed for regular quantization (based on  powers of norms) in~\cite{GrafL2000} (see also some recent extensions in~\cite{LuPag23})  for the unbounded setting. The first fully rigorous proof when the distribution is compactly supported is due to~\cite{BucklewW1982}. But their extension to non-compactly supported distribution  using {\em companders} contains an unsolved  gap. Filling this gap needs an extra  argument based on a random quantization argument known as  Pierce's Lemma developed in~\cite{GrafL2000}. For a recent more sophisticate  version  of Pierce's Lemma that we use in our proof, see~\cite[Theorem~5.2$(b)$]{PagSpring2018} or~\cite[Corollary~2.1.13$(b)$]{LuPag23}.

The main difficulty   to overcome is that Bregman divergences, when there  are not Euclidean norms, are not {\em isotropic} which implies to control the underlying function $F$ and $\phi_{_F}$ carefully in the support   of the distribution $P$ under consideration. Moreover,  of course, a Bregman divergence  does not satisfy the triangle inequality. These features  have  a major impact on the so-called  {\em firewall lemma} \textcolor{black}{(Lemma~\ref{lem:FirewallLemma})} which is the key to establish the ``lower   side'' of  our  Zador like theorem on the sharp convergence rate. Let us have in mind that the purpose of this lemma is to provide a somewhat minimal set of points to be added to a quantizer to locally control the nearest neighbour searches. We propose a refined  version of this firewall lemma, adapted to Bregman divergence.  


The paper is organized as follows. Section~\ref{sec:DefsVoro} is devoted to some short background on $L^r$-optimal ``regular'' quantization with a focus on the original Zador's Theorem as stated in~\cite{GrafL2000}(when the loss function is the power of a norm).
 Section~\ref{sec:ZadorBreg} is devoted to the proof of Zador's Theorem for $(r, \phi_{_F})$-optimal quantization w.r.t. a Bregman  divergence, under some appropriate assumptions (thus we request a sub-quadratic behaviour at infinity when the distribution to be quantized has an unbounded support). This lengthy section is divided in several steps which follow the structure of the  proof from~\cite{GrafL2000} of Zador's Theorem in the regular framework. An appendix is devoted to the proof of the key lemma of the  $\liminf$ side of the proof of the Theorem, that  is the {\em firewall lemma}. \color{black}{We conclude by Section~\ref{sec:Zadorfieldmatrix} devoted to the variant where we replace Bregman divergence by a (continuous and bounded) matrix valued field $x\mapsto S(x)\!\in {\cal S}^{++}(d, \R)$  leading to the  similarly measure $(\xi-x)S(x)(\xi-x)^{\top}$. Although we enhanced Bregman divergence on purpose in this chapter~--~and throughout the manuscript~--~ one may plead that, for applications, matrix fields  is a more convenient tool to introduce {\em anisotropy} in optimal vector quantization theory.
 \color{black}

\section{Definitions and background on $L^r$-optimal quantization w.r.t. a norm}\label{sec:DefsVoro}

\subsection{Short background on regular $L^r$-optimal vector quantization} 
 
Let $P$ be a probability measure on $(\R^d,{\cal B}or(\R^d))$ supported by a nonempty open convex set $U$  in the sense that $P(U)=1$.


\begin{definition}[Quantization error] Let  $r>0$ and let $\|\cdot\|$ denote any norm on $\R^d$. 

\noindent $(a)$ Let $\Gamma  \subset \R^d$ be a finite subset of $U$ (also called {\em quantizer}). We define the $L^r$-mean quantization error of the distribution $P$ induced by $\Gamma$ (with respect to the norm $\|\cdot\|$) by 
\begin{equation*}
    e_{r,n}(\Gamma,P,\|.\|) = \left[\int_{\mathbb{R^d}}\min_{a\in\Gamma} \|\xi-a\|^rP(\mathrm{d}\xi)\right]^{\frac{1}{r}}.
\end{equation*}

\noindent $(b)$ The  $L^r$-optimal mean quantization error for $P$ at level $n\ge 1$ is defined by 
\begin{equation*}
    e_{r,n}(P,\|.\|) = \inf_{|\Gamma|\leq n}  e_r(\Gamma,P,\|.\|).
\end{equation*}
\end{definition}

If one considers any random vector $X:(\Omega,{\cal A}, \P)\to U \subset  \R^d$ with distribution $\P_{_X}= P$, then
\[
   e_r(\Gamma,P,\|.\|) =    e_r(\Gamma,P,\|.\|) := \big\| \min_{a\in\Gamma} \|X-a\|\big\|_{L^r(\P)}= \left[ \E\, \min_{a\in \Gamma}\|X-a\|^r\right]^{1/r}
\]
and we may  often denote $e_{r,n}(X,\|.\|) =e_{r,n}(P,\|.\|) $ accordingly.
 
\bigskip
We will recall in Section~\ref{subsec:RecallExist} below  some theorems of existence of optimal quantizers for the Bregman divergence based $(r,\phi_{_F})$-quantization errors (those recalled, revisited  or extended in~\cite{BoutPag1}). In fact we establish several results depending on the quantization level ($n=1$ versus $n\ge 2$) and the power of the Bregman divergence ($r=2$ versus $r>2$).   Recalling these results is natural since one aim of optimal quantization, whatever the loss function is, is to produce optimal quantizers and evaluate their performances which are, for a given distribution $P$ and a given Bregman divergence $\phi_{_F}$, $n$ and $r\ge 2$ being fixed those of $e_{r,n}(P, \phi_{_F})$. However we will never make use of such $(r,\phi_{_F})$-optimal quantizers in the proofs (so that our ``\`a la Zador theorem'' holds for $r>0$).

\subsection{Sharp rate for $L^r$-optimal quantization : Zador's  theorem}


The so-called Zador Theorem is the main and central result in ``classical'' optimal quantization. It   elucidates in   great generality the sharp rate of decay  of the $L^r$-optimal mean quantization errors $e_{n,r}(P,\|\cdot\|)$ to $0$. This problem first tackled by Zador in his PhD thesis~(\cite{ZadorPhD})  in the 1960's (and finally published in~\cite{ZadorIEEE} in the late 1980's), essentially for the uniform distributions on the unit hypercube, then it was  extended by Bucklew \& Wise  to distributions having enough finite moments (but with a gap in the proof), see~\cite{BucklewW1982}. It was finally proved rigorously for the first time by Graf \& Luschgy in~\cite[Section 6.2]{GrafL2000}. We reproduce below for the reader convenience Graf \& Luschgy's version from~\cite[Section 6.2]{GrafL2000}.  In~\cite{GrafL2000}, the theorem is stated for $r\ge 1$ but a careful reading of the proof shows that it holds true for any $r\!\in (0,+\infty)$ as emphasized (and detailed in~\cite{LuPag23}). The result was again  generalized, only for some  specific classes of distributions sharing radial features in~\cite{LuPag23} for which it was proved that the integrability condition could be lowered down to the minimal finiteness of the $L^r$-moment for $L^r$-quantization.   In the present chapter we  extend in  Theorem~\ref{thm:ZadorBregman}$(a)$ the result from~\cite[Section 6.2]{GrafL2000} to the case where the {\em loss function}  is a Bregman divergence as defined in the next section. Claim~$(b)$ is in fact a by product of the proof of Claim~$(a)$ up to an application of Beppo Levi's theorem in both settings.

\begin{theorem}[Zador, Graf \& Luschgy 2000, Luschgy \&Pag\`es 2023]\label{thm:reg_zador} Let $r>0$ and let $\|\cdot\|$ denote any norm on $\mathbb{R}^d$. 

\smallskip
\noindent $(a)$  Assume $\int_{\mathbb{R}^d}\|\xi\|^{r+\delta}P(d\xi)< + \infty$ for some $\delta>0$.

\begin{equation}\label{eq:sharpZador}
    \lim_{n\rightarrow+\infty} n^{\frac 1d}e_{n,r}(P,\|.\|) = Q_r([0,1]^d,\|.\|)^{\frac 1r} \big \|h\big \|^{\frac 1r}_{L^{\frac{d}{d+r}}(\lambda_d)}
\end{equation}
where $h=\frac{\mathrm{d}P^a}{\mathrm{d}\lambda_d}$ denotes the density of the absolutely continuous part $P^a$ of $P$ with respect to the Lebesgue  measure $\lambda_d$ on $\mathbb{R}^d$. Furthermore
\begin{equation}
    Q_r([0,1]^d,\|.\|) := \inf_{n\geq1} n^{\frac rd} e_{r,n}([0,1]^d,\|.\|)^r = \lim_{n\geq1} n^{\frac rd} e_{r,n}([0,1]^d,\|.\|)^r  \!\in (0, +\infty)
\end{equation}
corresponds to the case $P= U\big([0,1]^d\big)$. 

When $P$ is radial in the sense that $P=P\circ O^{-1}$ (image of $P$ by $O$) for every orthogonal matrix $O$, the result is still true under of the moment assumption is only satisfied with $\delta =0$.

\smallskip
\noindent $(b)$ For any distribution $P$ on $(\R^d, {\cal B}or(\R^d))$,
\begin{equation}\label{eq:liminfZador}
    \liminf_{n\rightarrow+\infty} n^{\frac 1d}e_{n,r}(P,\|.\|) \ge Q_r([0,1]^d,\|.\|)^{\frac 1r} \big \|h\big \|^{\frac 1r}_{L^{\frac{d}{d+r}}(\lambda_d)}
\end{equation}

\noindent $(c)$ When $d=1$  and $P=U([0,1])$, then for every $r>0$, the $L^r$-optimal quantizer is unique and 
\[
\Gamma^{(r,n)}= \Big\{\frac{2k-1}{2n}, \; k=0,\ldots,n\Big\}.
\]
and, for every $n\ge 1$,
\color{black}
\[
e_{r,n}\big(U[(0,1]),|\cdot|)= \frac{1}{2n(r+1)^{1/r}} \quad \mbox{ so that }\quad Q_r([0,1],|\cdot|)= \frac{1}{2^r(r+1)^{1/r}}.
\]
 \color{black}
 \end{theorem}
%
\noindent {\sc Notation.} To alleviate notation,  when the norm $\|\cdot\|= |\cdot|$ is the canonical Euclidean norm on $\R^d$, we will denote
\begin{equation}\label{eq:Q_r}
Q_r([0,1]^d)=Q_r([0,1]^d,|\cdot|).
\end{equation}

\smallskip
 \noindent  {\bf Remarks} $\bullet$ When  $d=2$  and $P=U([0,1]^2)$, then (see~\cite{Newman1982})
 \[
 Q_2([0,1]^2) = \frac{5}{18\sqrt{3}}.
 \]
which is closely connected with the tiling of $\R^2$ by regular Hexagons.  When  $d=3$, the fact that  
 \[
 Q_2([0,1]^3) = \frac{19}{64\sqrt[3]{2}}
 \]
 still stands as a conjecture as our best knowledge and corresponds to the tiling of \textcolor{black}{$\R^d$} by truncated octahedrons.
 
 \smallskip
 \noindent $\bullet$  If the absolutely continuous part $P^a$ is zero then Zador's theorem still holds as it is (with $h\equiv 0$). This shows that the proposed asymptotics is not the right one in the sense the one where the limit, if any, is non-trivial.

 \smallskip
\noindent $\bullet$ It is proved still in~\cite[Remark~6.3 of Section 6.2]{GrafL2000} that 
$$
\int_{\mathbb{R}^d}|\xi|^{r+\delta}P(\mathrm{d}\xi) <+\infty \;\Longrightarrow\; h\!\in L^{\frac{d}{d+r}}(\lambda_d)
$$
Indeed this  is a consequence of H\"older inequality applied to 
\[
h(\xi)^{\frac{d}{d+r}}= (1+|\xi|)^{-\frac{r+\delta}{d+r}d}\times  (1+|\xi|)^{\frac{r+\delta}{d+r}d}h(\xi)^{\frac{d}{d+r}}
\]
with conjugate exponents $p= 1+\frac dr $ and $q= 1 +\frac rd$ so that
\[
\int_{\R^d}h(\xi)^{\frac{d}{d+r}}\mathrm{d}\xi \le \left(\int_{\R^d} (1+|\xi|)^{-\frac{r+\delta}{r}d}\mathrm{d}\xi\right)^{\frac{r}{d+r}}\left( \int_{\R^d}(1+|\xi|)^{r+\delta}h(\xi) \mathrm{d}\xi\right)^{\frac{d}{d+r}}<+\infty.
\]

\noindent $\bullet$ The (easy) case $0<r<1$ is detailed in~\cite{LuPag23}.

\noindent $\bullet$  In~\cite[Theorem~2.1.3, Chapter~2]{LuPag23}, it is shown that regardless of the integrability condition, one always has
\begin{equation*}\label{eq:liminfZadorbis}
    \liminf_{n\rightarrow+\infty} n^{r/d}e_{n,r}(P,\|.\|)^r \ge  Q_r([0,1]^d)\big \| h\big \|_{L^{\frac{d}{d+r}}(\lambda_d)}.
\end{equation*}

\noindent 
$\bullet$  Finally, one can also prove (see~\cite{LuPag23}) that, if $P$ is radial  outside a compact set  i.e. $P= h(\xi) \lambda_d(\mathrm{d}\xi)$ with $h(\xi)= g(\|\xi\|_0)$, $\|\xi\|_0>A$ for some $A\ge 0$, then the above sharp quantization rate of decay holds under  the less stringent  assumption that $P$ has a finite (absolute)  moment of order $r>0$.


\smallskip
\noindent $\bullet$  In view of what follows with Bregman divergence it is interesting to inspect a variant of the above result: if the norm under consideration is an Euclidean norm denoted $|\cdot|_{e}$ and $P$ is supported by a (nonempty) closed convex set, say $C$, then one derives form the well-known fact that the resulting projection  ${\rm Proj}_{C}$ on $C$  is $|\cdot|_{e}$-Lipschitz with Lipschitz coefficient $1$ \textcolor{black}{and} satisfies the following inequality that holds any quantizer $\Gamma\subset \R^d$,
\[
e_r(\Gamma, P, |\cdot|_{e}) \le  e_r(\{{\rm Proj}_{C}(a), \, a\!\in \Gamma\}, P, |\cdot|_{e})
\]
since $|\xi-  {\rm Proj}_{C}(a)|_{e}\le |\xi- a|_{e}$ for every $\xi\!\in C$ and every $a\!\in\R^d$. Hence, one easily checks that  
\begin{align*}
e_{n,r}(P, |\cdot|_{e})^r &= \inf \big\{ e_r(\Gamma, P, |\cdot|_{e})^r : \Gamma\subset \R^d, \; |\Gamma|\le n\big\}\\
&=  \inf \Big\{ \int \min_{b\in \Gamma}|\xi-b|^r P(\mathrm{d}\xi) : \Gamma\subset C, \; |\Gamma|\le n\Big\}.
\end{align*}
Now assume that $U=\mathring{C}\neq \varnothing$ and  that $P(U)=1$.
Another argument based on consequence of the support hyperplane theorem (see among other~\cite{Pages1998},~\cite{GrafL2000} or~\cite{LuPag23})  shows  that an $L^r$-optimal quantizer is then always $U$-valued in that situation.

\subsection{Optimal vector quantization with respect to a Bregman divergence: definitions and first properties} 
  \subsubsection{Bregman divergence} 
  First we introduce the notion of Bregman divergence induced by a strictly convex, continuously (Fr\'echet-)differentiable function $F$ defined on a nonempty convex open set $U$ of $\R^d$.
\begin{definition}[Bregman divergence associated to strictly convex function $F$]\label{def:Bregdivchap2}
Let $F: U \rightarrow \R$ be a continuously differentiable, strictly convex function  defined on a nonempty convex open set $U$ of $\R^d$.
The Bregman divergence induced by $F$ of $\xi\in U$ with respect  to $x\in U$  is defined by
$$
    \phi_{_F}(\xi,x) = F(\xi)-F(x) -\langle \nabla F(x)\,|\, \xi-x\rangle\geq 0,
$$
where denotes $\nabla F(x)$ the gradient vector of $F$ evaluated at $x$. Note that $\phi_{_F}(\xi,x)=0$ if and only if $u=v$ since $F$ is strictly convex.
\end{definition}
When no ambiguity we will often denote $\phi(\xi,x)$ instead of $\phi_{_F}(\xi,x)$.

\medskip 
\noindent {\bf First properties of interest.} $\bullet$ When $F$ is twice (Fr\'echet-)differentiable on $U$ then a second order Taylor expansion with integral remainder yields the alternative formulas that we will extensively use
\begin{equation}\label{eq:bregmadivFC2bis}
   \phi_{_F}(\xi,x)= \int_0^1 (1-u)\nabla^2 F(x+u(\xi-x)).(\xi-x)^{\otimes 2} \mathrm{d}u=  \int_0^1 v\nabla^2 F(\xi+v(x-\xi)).(x-\xi)^{\otimes 2} \mathrm{d}v.
\end{equation}

\noindent $\bullet$ Bregman divergences does not satisfy the triangle inequality nor the symmetric property, which makes it significantly different from a distance.
    For $\xi,x,y\! \in U$, {\em in general}, 
    \begin{itemize}
    
    \smallskip
        \item $\phi_{_F}(\xi,x) \nleq \phi_{_F}(\xi,x) + \phi_{_F}(x,y)$,
        
        \smallskip
        \item $\phi_{_F}(\xi,x) \neq \phi_{_F}(x,\xi)$.
    \end{itemize}
  
  \section{Introduction to quantization with respect to a Bregman divergence}  The idea of Bregman quantization is to replace the norm  which plays the role of a loss function in regular vector quantization theory  by the   Bregman divergence $\phi_{_F}$ of a continuously  differentiable  strictly convex function $F$ as defined  in~\eqref{eq:bregmadivFC2bis}. 
  
  The  function $F$ being defined on an open set $U\subset \R^d$, we will consider in what follows $U$-valued quantizers in our definitions of the quantization  error. This choice, although natural,  may appear somewhat arbitrary when the support of $P$ is strictly included in $U$ and, except in $1$-dimension, it seems unclear that it  has no influence on the  quantization error, e.g. compared to another natural choice like  the convex envelope of the support of the distribution $P$. However this choice makes problem at least in higher dimensions, because the natural domain of definition of convex differentiable function defined on $\R^d$ is an open convex set.    
        
\begin{definition}[Quantization Error w.r.t.   Bregman divergences]\label{def:BregQerror} 
Let $r\!\in (0, +\infty)$.  Let $P$ be a probability distribution supported by $U$ and let $X:(\Omega,{\cal A}, \P)\to \R^d$ be a $P$-distributed random vector.

\noindent $(a)$  Let $\Gamma\subset U$ be  a finite subset of $U$ (also called {\em quantizer}). The $(r, \phi_{_F})$-mean quantization error for Bregman divergences $\phi_{_F}$ of the distribution $P$ is defined by 
\begin{equation}\label{eq:meanBregQerror}
    e_{r,n}(\Gamma, P,\phi_{_F}) =  \left[\int_U \min_{a \in \Gamma} \left (\phi_{_F}(x,a)\right )^{\frac r2} \mathrm{d}P(x) \right]^{\frac{1}{r}}= \left[ \E \, \min_{a\in \Gamma}\phi_{_F}(X,a)^{\frac r2}\right]^{\frac 1r}.
\end{equation}

\noindent $(b)$ The optimal $(r, \phi_{_F})$-mean quantization error  at level $n\!\ge 1$ is defined by
\begin{equation}\label{eq:OptimeanBregQerror}
    e_{r,n}(P,\phi_{_F}) = \inf_{ |\Gamma|\leq n}  e_{r,n}(\Gamma, P,\phi_{_F}).
\end{equation}
\end{definition}

We will use indifferently    $e_{r,n}(\Gamma, X,\phi_{_F})$ and  $e_{r,n}(X,\phi_{_F})$ to denote the above quantities.

\medskip
\noindent {\bf Comment-Warning}. The terminology  {\it $(r,\phi_{_F})$-mean quantization error} has been assigned to this error modulus to fit with the usual terminology when $F(x) = |x|^2$ (Euclidean norm) since then $\phi_{_F}(\xi,x)= |\xi-x|^2$. We are conscious that it  may induce a notational confusion since one has \textcolor{black}{$e_{r,n}(P, |\cdot|) = e_{r,n}(P, \phi_{|\cdot|^2})$ where $e_{r,n}(P, |\cdot|)$ stands for the standard notation in regular quantization based on  (powers) of norm) as similarity measures}.

\begin{proposition}[Integrability]  Let $r>0$. If  $\mathbb{E}\big(|F(X)| \vee |X|\big)^{\frac r2}  < + \infty $, then for every $x\!\in U$, $\mathbb{E}\big(|\phi_{_F}(X,x)|^{\frac r2} )< +\infty$.
\end{proposition}
\noindent {\em Proof.} One has 
\begin{align*}
        \mathbb{E}\,|\phi_{_F}(X,x)|^{\frac r2 }&= \int_U |F(\xi) -F(x) - \langle\nabla F(x)| \xi-x\rangle | ^{\frac r2}P(\mathrm{d}\xi) \\
        & \le(1+|\nabla F(x)| )^{\frac r2} \int_U \big(|F(\xi) | \vee  | \xi | \big)^{\frac r2}P(\mathrm{d} \xi)  \\
    \hskip 5cm     &\quad + |F(x)| +|\langle \nabla F(x) | x\rangle | <+ \infty. \hskip 5cm \Box
    \end{align*}
\begin{proposition}\label{property:2chap2}
    $(a)$ Assume that $F:U \to \R$ has a Lipschitz continuous gradient $\nabla F$ on $U$. Then, for every  $\xi,\,x\!\in U$,
    \begin{equation}\label{eq:ComparEuclide}
        0 \leq \phi_{_F}(\xi,x) \leq \tfrac 12  [\nabla F]_{\rm Lip}|\xi- x|^2.
    \end{equation}
    Also note that under this assumption $F$ and $\nabla F$ can be continuously extended to $\bar U$.
    
    \smallskip
    \noindent $(b)$ If $F$ is $\alpha$-convex for some $\alpha >0$ in the sense $\nabla F $ exists and is $\alpha$-coercive on $U$:
    \[
   \forall\, \xi, x\!\in U, \quad  \langle \nabla F(\xi ) - \nabla F(x) | \xi- x \rangle \ge \alpha |\xi- x|^2,
    \]
    then 
    \[
     \phi_{_F}(\xi,x) \ge  \frac{\alpha}{2}  |\xi- x|^2.
     \]
     \end{proposition}
\noindent {\em Proof.} $(a)$ Let $\xi, x\!\in U$. One has using Cauchy-Schwartz inequality in the second line, that 
    \begin{align*}
        \phi_{_F}(\xi,x) &= F(\xi) - F(x) -\langle \nabla F(x) | \xi-x \rangle \\
       & =\int_0^1  \langle \nabla F((1-u)\xi +u x) -\nabla F(x)\,|\, \xi-x\rangle du\\
    & \le \int_0^1  \big|\nabla F((1-u)\xi +u x) -\nabla F(x)\big|| \xi-x| du\\
    & \le  \int_0^1  (1-u)du [\nabla F]_{\rm Lip} | \xi-x|^2  =  \tfrac 12 [\nabla F]_{\rm Lip}|\xi-x |^2.
    \end{align*}
    
    The second claim is a straightforward consequence of  Cauchy's criterion.
    
    \noindent $(b)$ is an easy consequence of the fact that the function
    \[
    g(t)  = F(t\xi+(1-t)x)-F(x) -t\langle \nabla F(x)\,|\, \xi-x\rangle -\frac{\alpha}{2}t^2|\xi-x|^2
    \]
     has a nonnegative derivative on $(0,1]$ since 
     \begin{align*}
     g'(t) &= \langle \nabla F(t\xi+(1-t)x)-\nabla F(x)\,|\,\xi-x\rangle -\alpha t|\xi-x|^2\\
     & =\frac 1t \Big(\big\langle \nabla F(t\xi+(1-t)x)-\nabla F(x)\,|\,t(\xi-x)\big\rangle -\alpha \big|t(\xi-x)\big|^2\Big)\ge 0.
     \end{align*}
  This  implies  $g(1)\ge g(0)$. 
     \hfill $\Box$

\subsection{Existence of optimal quantizers (background)}\label{subsec:RecallExist}
We recall here the existence theorems for the  existence of optimal quantizers $(r,\phi_{_F})$-Bregman divergence  established in~\cite{BoutPag1} (although in the proofs that follow we will always manage not to use these existence results).

\medskip 
\noindent {\bf Preliminary remark} ({\em Shrinking may help}).The nonempty open set $U$ is defined as the definition domain of the function $F$ in the above definition~\ref{def:Bregdivchap2} and the distribution $P$ is supposed to satisfy $P(U)=1$. However in many applications, some assumptions below such as ~\eqref{eq:boundary-lsc-bis-chap2} or~\eqref{eq:Cond-ra-bis-chap2}--\eqref{eq:Cond-rb-bis} are both satisfied owing to the behaviour of the Bregman divergence  outside $U\times U$.  In fact it is clear that the conclusion of the theorems below still hold if there exists an open convex set $V\subset U$ such that $P(V)=1$ for which $F_{|V}$ and $\phi_{F_{|V}}= (\Phi_{_F})_{V\times V}$ satisfy these assumptions. With, as a by-product, that the optimal $n$-quantizers in $x^{(n)}$ will be $V^n$-valued since $V$ plays the  role of $U$ in  the theorems below. Thus, Condition~\eqref{eq:boundary-lsc-bis-chap2}}  may be easier to satisfy e.g. if $V$ is bounded while $U$ is not.

We introduce for convenience the $(r, \phi_{F})$-distortion function $G_{r,n}:U^n\to \R_+$ as $r$-th power of the $(r, \phi_{F})$-mean quantization error i.e.
\begin{align*}
\forall\, x=(x_1,\ldots,x_n)\!\in U^n, \;G_{r,n}(x_1,\ldots,x_n)&= e_r\big(\{x_1, \ldots,x_n\},P\phi_{_F}\big)^r\\
= \E\, \min\phi_{_F}(X,x_i)^r&= \int_U \min\phi_{_F}(\xi,x_i)^rP({\rm d}\xi).
\end{align*}
Note that one clearly has
\[
\inf_{x\!\in U^n}G_{r,n}(x_1,\ldots,x_n)=\Big( \inf_{\Gamma\subset U, \, |\Gamma|\le n}e_r\big(\{x_1, \ldots,x_n\},P,\phi_{_F}\big)\Big)^r.
\]
since any grid $\Gamma$ with at most $n$ element, one may build a $n$-tuple $x$ such that $\Gamma=\{x_1, \ldots,x_n\}$ by repeating some components.
 
\begin{theorem} [Existence when $r=2$] \label{thm:Existence-bis}   Assume that the distribution $P$ satisfies  $\mathbb{E}\big(|F(X)| \vee |X|\big) < + \infty $  and,  
if $U$ is unbounded and $d\ge 2$, that the $F$-Bregman divergence $\phi_{_F}$ satisfies
\begin{equation}\label{eq:boundary-lsc-bis-chap2}
\forall \xi \!\in U
, \quad \liminf_{|x|\to +\infty,\, x\in U}\phi_{_F}(\xi, x)= \sup _{x\in U} \phi_{_F}(\xi,x).
\end{equation} 


\medskip
\noindent $(a)$ Then for every $n\ge1$ there exists an $n$-tuple $x^{(n)}= \big(x^{(n)}_1, \ldots, x^{(n)}_n\big) \!\in U^n$ which minimizes $G_{n}$ over $U^n$. Moreover, if  the support (in $U$) of the distribution $P$ has at least $n$ points then $x^{(n)}$ has pairwise distinct components and $P\big(\mathring C_i(x^{(n)})\big)>0$ for every $i\!\in \{1,\ldots,n\}$.

\smallskip
\noindent $(b)$ The distribution  $P$ assigns no mass to the   boundary of Bregman-Voronoi partitions of $x^{(n)}$ i.e.
\[
P\left (\bigcup_{i=1}^n\partial C_i\big( x^{(n)})\right) = 0
\]
and $x^{(n)}$ satisfies the stationary (or master) equation
\begin{equation}\label{eq:Mastereqbis}
P\big( C_i(x^{(n)})\big)\,x^{(n)}_i  -  \int_{ C_i(x^{(n)})} \xi P(\mathrm{d}\xi) =0, \quad =1,\ldots,n.
\end{equation}
 
\noindent $(c)$ The sequence $\displaystyle G_{n}\big( x^{(n)}\big) = \min_{U^n} G_{n}$ decreases as long as it is not $0$ and converges to $0$ as $n$ goes to $+\infty$.

\noindent $(d)$ When $d=1$, the above claims $(a)$-$(b)$-$(c)$ remain true without assuming~\eqref{eq:boundary-lsc-bis-chap2}.
\end{theorem}

In the  theorem below $\bar U^{\widehat{\R^d}}$ denotes the closure of $U$ in the Alexandroff compactification of $\R^d$.

\begin{theorem}[Existence when $r>2$] \label{thm:Existence-r-bis}   
Let $r\!\in (2, +\infty)$. Assume the distribution $P$ of $X$ satisfies the  moment assumption $\mathbb{E}\,\big(|X| \vee |F(X)| \big)^{\frac  r2} <+\infty$
and $F$ is $C^2$ on $U$ with $\nabla^2F(x)$ (symmetric) positive definite for  every $x\!\in U$.
%
%
%
%
Assume that at every $\bar x \!\in \hat\partial  \bar U^{\widehat{\R^d}}$,\\ either
\begin{equation} \label{eq:Cond-ra-bis-chap2}
(i)\;\nabla^2F \mbox{ can  be continuously extended at } \bar x \mbox{ and } \nabla^2F(\bar x) \mbox{ is (symmetric) positive definite}
\end{equation}
or  

$(ii)$\; the l.s.c. extensions $\phi_{_F}(\xi, \cdot)$  on $ \bar U^{\widehat{\R^d}}$ satisfy  
\begin{equation}\label{eq:Cond-rb-bis}
 \forall\, \xi \!\in U, \qquad \phi_{_F}(\xi, \bar x) =\displaystyle  \sup_{x\in U}\phi_{_F}(\xi,x)
\end{equation}
where, if $U$ is unbounded, $\bar x = \infty$ satisfies the above condition~$(ii)$.

\medskip
\noindent $(a)$ Then for every $n\ge1$ there exists an $n$-tuple $x^{(r,n)}= \big(x^{(r,n)}_1, \ldots, x^{(r,n)}_n\big) \!\in U^n$ which minimizes $G_{n}$ over $U^n$. Moreover, if  the support (in $U$) of the distribution $P$ has at least $n$ points then $x^{(r,n)}$ has pairwise distinct components and $P\big(\mathring C_i(x^{(r,n)})\big)>0$ for every $i\!\in \{1,\ldots,n\}$.

\smallskip
\noindent $(b)$ The distribution  $P$ assigns no mass to the   boundary of Bregman-Voronoi partitions of $x^{(n)}$ i.e.
\[
P\left (\bigcup_{i=1}^n\partial C_i\big( x^{(r,n)})\right) = 0
\]
and $x^{(n)}$ satisfies the $(r,\phi_{_F})$-stationary (or $(r,\phi_{_F})$-master) equation
\begin{equation}\label{eq:rMastereqbis}
 x^{(r,n)}_i = \frac{ \int_{C_i(x^{(r,n)})}\xi \phi_{_F}( \xi,x^{(r,n)}_i )^{\frac r2-1}P(\mathrm{d}\xi)}{\int_{C_i(x^{(r,n)})}\phi_{_F}( \xi,x^{(r,n)}_i )^{\frac r2-1}P(\mathrm{d}\xi) }  , \quad =1,\ldots,n.
\end{equation}
 
\noindent $(c)$ The sequence $\displaystyle G_{r,n}\big( x^{(r,n)}\big) = \min_{U^n} G_{r,n}$ decreases as long as it is not $0$ and converges to $0$ as $n$ goes to infinity.
\end{theorem}


\noindent {\bf Remarks.} $\bullet$ When $n=1$, i.e. the  case of the Bregman {\em $(r,\phi_{_F})$-median}, the assumptions on $F$ and $P$ can be slightly relaxed (see~\cite{BoutPag1}).

\smallskip
\noindent $\bullet$  Note that the assumptions on $F$ are more stringent when $r>2$ compared to the case $r=2$.

\bigskip
\noindent {\bf Examples of Bregman divergences in one dimension.} 
These examples are all  $1$-dimensional so that Theorem~\ref{thm:Existence-r-bis} applies without further  conditions on the functions $F$.
\begin{enumerate} 
\item {\em Regular quadratic loss function}. $F(x)= x^2$, $U= \R$,  $\phi_{_F}(\xi,x) = (\xi-x)^2$ 

\item {\em Norm--like}. $F(x)= x^a$, $a>1$, $U= (0, +\infty)$, $\phi_{_F}(\xi,x)= \xi^a +(a-1)x^a- a \,\xi \,x^{a-1}$,
\item {\em Itakura--Saito divergence.} $F(x)= -\log (x)$, $U= (0, +\infty)$, $\phi_{_F}(\xi,x) = \log \big(\frac x\xi\big)+\frac \xi x-1$.

\item {\em I-divergence or Kullback-Leibler divergence}.  $F(x) = x\log(x)$, $U= (0, +\infty)$, $\phi_{_F}(\xi,x)= \xi\Big( \log\big(\frac \xi x\big)-1+\frac x\xi \Big)$.

\item {\em Logistic loss}. $F(x)= x\log x +(1-x)\log(1-x)$,  $U= (0, 1)$, $\phi_{_F}(\xi,x)=\xi \log\frac \xi x +(1-\xi)\log\Big(\frac{1-\xi}{1-v}\Big)$.

\item {\em Softplus loss}. $F(x)= \log(1+e^{x})$, $U= \R$, $\phi_{_F}(\xi,x) = \log\Big(  \frac{1+e^\xi}{1+ e^x}\Big) - \frac{e^x}{e^x +1}(\xi-x)$.

\item {\em Exponential loss}. $F_\rho(x)= e^{\rho x}$,   $\rho\!\in \R$, $U= \R$, $\phi_{F_\rho}(\xi,x) = \phi_{F_1}(\rho \xi,\rho x)$ with \linebreak $ \phi_{F_1}(\xi,x) = e^\xi -e^x-e^v(\xi - x)$. 
 \end{enumerate} 

\medskip 
\noindent {\bf Remark.} Note that  the function $F$ in 6. and 7.  does not fulfill the assumption contained e.g. in~\cite{Fischer2010} to guarantee existence of optimal quantizers for such Bregman divergences, that is 
\begin{equation}\label{eq:Cond2b-chap2}
\forall\,  \bar x \!\in \partial \bar U^{\widehat{\R^d}},\; \forall\,  \xi \!\in U, \quad  \liminf_{x\to \bar x, \,x\in U} \phi(\xi,   x) =  \sup_{x\in U} \phi(\xi,x).
\end{equation}

 \bigskip
\noindent {\bf Examples of Bregman divergence in higher dimension.} 
 \begin{enumerate} 
 \item {\em Regular quadratic loss function}. $F(x)= |x|^2$, $U= \R^d$.
 \item {\em Mahalanobis distance}. $F(x)= x^*Sx$, $S\!\in {\cal S}^+(d,\R)\cap GL(d,\R)$ (symmetric and positive definite), $U= \R^d$ and  $\phi_{_F}(\xi,x)= (\xi-x)^*S(\xi-x):=|\xi-x|^2_{_S}$. (Note that $F$ is simply the square of an Euclidean norm so that this case is already contained in classical optimal quantization theory with $U= \R^d$.)
\item {\em $f$-marginal divergence (additive)}. $F(x_1,\ldots,x_d)= \sum_{i=1}^d f(x_i)$, $f$ strictly convex on $U$, is defined on $U^n$, $\phi_{_F}(\xi,x)=  \sum_{i=1}^d \phi_{_F}(\xi_i,x_i)$.
  \item {\em $f$-marginal divergence (multiplicative)}. $F(x_1,\ldots,x_d)= \prod_{i=1}^d f(x_i)$, $f$ strictly convex on $U$, is defined on $U^n$, $\phi_{_F}(\xi,x)=  \sum_{i=1}^d \phi_{_F}(\xi_i,x_i)\prod_{j\neq i}\phi_{_F}(\xi_i,x_j)$.
  \item {\em Soft max marginal $f$-divergence}.  $F(x_1,\ldots,x_d)=F_{\lambda}(x_1,\ldots,x_d)= \frac{1}{\lambda}\log\Big(\sum_{i=1}^d e^{\lambda f(x_i)}\Big)$, $f$ strictly convex on $U$, is defined on $U^d$ and for every $x=(x_1,\ldots,x_d)$, $\xi=(\xi_1, \ldots,\xi_d)\!\in U^d$,
  $$
  \phi(\xi,x)= F_{\lambda}(\xi)-F_{\lambda}(x)-  \lambda \frac{\sum_{1\le i\le d} f'(x_i)e^{\lambda f(x_i)}(\xi_i-x_i)}{\sum_{1\le i\le d} e^{\lambda f(x_i)}}.
  $$ 
  \end{enumerate}  


\bigskip
\noindent {\bf Warning (bis) !} There is a conflict of notation between the  regular $(r,\phi_{_F})$-mean quantization errors associated to the  Euclidean norm, namely  $e_r(\Gamma, P, |\,\cdot\,|)$ and  $e_{r,n}(P, |\,\cdot\,|)$ in regular optimal quantization theory and those associated to the Bregman divergence induced by this Euclidean norm  since $\phi_{_F} = |\,\cdot\,|^2$ when $F=|\,\cdot\,|^2$. In what follows we adopt the second notation  i.e.
\[
e_r(\Gamma, P,|\,\cdot\,|^2)= \left(\int_{\R^d}\min_{b\in \Gamma}|\xi-b|^rP(\mathrm{d}\xi)\right)^{\frac 1r}
\]
following Definition~\ref{def:BregQerror}. 

\nopagebreak\section{Asymptotic analysis of the quantization error: a Zador like theorem}\label{sec:ZadorBreg}
Proving  rigorously a Zador like theorem in the framework of Bregman divergence    will require several steps but, prior to this long way to the target, let us  note that combining the above classical Zador's  theorem and \textcolor{black}{Proposition~\ref{property:2chap2}$(a)$} that if $F$ is Lipschitz (and even simply convex) then, for any quantizer $\Gamma\subset U$,  
\[
e_{r}(\Gamma, P,F) \le \sqrt{\frac{[\nabla F]_{\rm Lip}}{2}} \,e_{r}(\Gamma, P, |\cdot|^2)
\]
where $|\cdot| $ denotes the canonical Euclidean norm. This  implies that, under the assumptions of the above Zador theorem (applied with the  Euclidean norm), one has
\[
 \limsup_{n\rightarrow+\infty} n^{1/d}e_{r,n}(P, F) \le \sqrt{\frac{[\nabla F]_{\rm Lip}}{2}} Q_r([0,1]^d)\big \| h\big \|_{L^{\frac{d}{d+r}}(\lambda_d)}.
 \]

Similarly if $F$ is $\alpha$-convex, Property~\ref{property:2chap2}$(b)$  implies 
\[
e_{r,n}(\Gamma, P,F) \ge \sqrt{\frac{\alpha}{2}} e_{r,n}(\Gamma, P, |\cdot|^2)
\]
so that
\[
 \liminf_{n\rightarrow+\infty} n^{r/d}e_{n,r}(P,\|.\|)^r  \ge \sqrt{\frac{\alpha}{2}}  Q_r([0,1]^d)\big \| h\big \|_{L^{\frac{d}{d+r}}(\lambda_d)}.
\]

\subsection{Zador's like Theorem for Bregman divergence}

Still few notational preliminaries needed to  state the theorem. First, we introduce    for every $\ve>0$ its (open) ``$\ve$-interior'' defined by 
\[
U_{\ve} = \big\{x\!\in U: d(x,U^c)>\ve\big\}
\] 
(the distance $d$ is taken w.r.t. to the canonical Euclidean norm). These {\em $\ve$-interiors}   of $U$ are non empty for small enough $\ve $ since $U$ is non empty and open.  Its closure   satisfies  
$$
\bar U_{\ve}\subset \big\{x\!\in U: d(x,U^c)\ge \ve\big\}\subset \big\{x\!\in U: d(x,U^c)> \ve'\big\}\subset U_{\ve'},\; \ve'<\ve.
$$

We recall that the notation ${\rm supp}(P)$ stands for the support of the distribution $P$ {\em in $\R^d$}. 
\begin{theorem}[Zador like theorem for Bregman divergence]\label{thm:ZadorBregman}
Let  $U\subset \mathbb{R}^d$ be a nonempty open convex subset of $\mathbb{R}^d$, let $F : U\subset \mathbb{R}^d \rightarrow \mathbb{R}$ be a $C^2$ strictly convex function such that 
\[
 \forall\, x\!\in U, \quad \nabla^2F(x) \!\in{\cal S}^{++}(d,\mathbb{R}):={\cal S}^{+}(d,\mathbb{R})\cap GL(d,\R).
\]

\noindent $(a)$  \textcolor{black}{{\em Asymptotic sharp rate}}. Let  $P$ a probability distribution supported by $U$  i.e. $P(U)=1$ such that 
\begin{equation}\label{eq:rintegZador}
\int_U\big(|\xi|^{2(1+\delta)}\vee |F(\xi)|\big)^{\frac{r}{2}}P(\mathrm{d}\xi)<+\infty\quad\mbox{ for some } \quad \delta>0.
\end{equation}
Let  $h=\frac{\mathrm{d}P^a}{\mathrm{d}\lambda_d}$ denote the density of the absolutely continuous part $P^a$ of $P$ with respect to the Lebesgue  measure $\lambda_d$ on $\mathbb{R}^d$ and 
\begin{equation}\label{eq:QrF}
Q_r(\phi_{_F},P) =  \textcolor{black}{2^{-r/2}}Q_r([0,1]^d)\big\|\det(\nabla^2 F)^{\frac{r}{2d}}\cdot h \big\|_{\frac{d}{d+r}}
\end{equation} 
where $Q_r([0,1]^d)$ is given by~\eqref{eq:Q_r}.

If
\begin{equation}\label{eq:HypoZadorBreg-chap2}
\left\{\begin{array}{ll}
(i) & \hskip-0,5cm {\rm supp}(P)\mbox{ is compact and included in $U$}\\
\mbox{or }\qquad &\\
(ii)&\hskip-0,5cm  \exists \,\eta \ge 0 \; \mbox{ s.t. }\; {\rm supp}_U(P)\subset U_{\eta} \;\mbox{ and }\; \sup_{x\in U_\eta}\vertiii{\nabla F^2(x)}<+\infty
\end{array}\right.
\end{equation}
 (with the convention  $U_0= U$) then
\begin{equation*}
    \lim_{n\rightarrow+\infty} n^{\frac1d} e_{r,n}(P,\phi_{_F}) = Q_r(\phi_{_F},P)^{\frac 1r}.
\end{equation*}
%
%
\smallskip
\noindent $(b)$ \color{black}{\em Universal lower bound}. For any distribution $P$ supported by $U$ one has
\begin{equation*}
    \liminf_{n\rightarrow+\infty} n^{\frac 1d } e_{r,n}(P,\phi_{_F}) \ge Q_r(\phi_{_F},P)^{\frac 1r}.
\end{equation*}
\color{black}
\end{theorem}

\noindent {\bf Remarks.} $\bullet$  {\em Shrinking $U$ may again help}. Like for the existence of optimal quantizers (see above Section~\ref{subsec:RecallExist}), the {\em Shrinking may help} trick is again a way to weaken the above boundedness assumptions on $\nabla^2F$. It allows to replace $U$ by a   convex subset $V$ of $U$ such that
\[
\sup_{x\in V}\vertiii{\nabla F^2(x)}<+\infty.
\]
The adaptation of the proof is almost a tautology  since it boils down to replace $F$ defined on $U$ to it restriction $F_{|V}$ defined on $V$. 
This maybe significantly  more general than the above assumption which corresponds to choose $V=U_{\eta}$ when $\eta>0$ (having in mind that $U_{\eta}$ is convex).

\smallskip
\noindent This version of the theorem leaves open the case where $F$ goes to infinity at some boundary points of $U$ in $\R^d$ (when $U\neq \R^d$) which belong to the support of $P$ in $\R^d$. As can be derived from Claim~$(b)$ a necessary condition for the sharp asymptotic rate at rate $n^{-1/d}$ to hold true is that 
\[
 \big\|\det(\nabla^2 F)^{\frac{r}{2d}}\cdot h \big\|^{\frac 1r}_{\frac{d}{d+r}}<+\infty.
 \]
But what happens when $\nabla^2 F$ is not bounded at infinity at a point lying in ${\rm supp}(P)\cap \partial U$ is not solved at all by this theorem. This open problem is of interest   when dealing with Bregman divergences induced by functions $F$ defined on $U=(0, +\infty)$ like, among others,
\[
F(x)  =-\log(x) \quad \mbox{ or} \quad F(x)= x\log(x)
\]
when $0\!\in {\rm supp}(P)$.

\smallskip
\noindent  $\bullet$ We managed to never use optimal quantizers  throughout the proof of this avatar of Zador's Theorem. So it may hold even if no optimal quantizers can be proved to exist. Of course the assumptions remain similar. 

\color{black}{\smallskip
\noindent  $\bullet$ As concerns the comparison with  the result stated in~\cite{LiuBelNIPS2016} in terms of assumptions, here is what we can say: our Assumption~\eqref{eq:HypoZadorBreg-chap2} is clearly lighter than the global  uniform continuity assumption of $\nabla^2F$ made  in~\cite{LiuBelNIPS2016} when $U$ is unbounded  since we only assume continuity of $\nabla^2 F$ through our $C^2$-assumption on $F$.  On the other hand we essentially assume that $\nabla^2 F$ is bounded on $U$. This assumption as set does not appear in~\cite{LiuBelNIPS2016}. However, the theorem is stated under a (power) moment assumption which implies in a somewhat hidden way that $F$ has at most quadratic growth (which seems not to be mentioned). 
We also deal in details with distributions possibly  having a singular component which does not appear clearly either in (the extended version of)~\cite{LiuBelNIPS2016}.
\color{black}

\bigskip
\noindent {\bf Practitioner's corner.} A Bregman divergence being given, we give conditions on the support of distributions  $P$ on $(\R^d, {\cal B}or(\R^d))$ satisfying  the integrability condition~\eqref{eq:rintegZador} for which Zador's Theorem applies.

\medskip
$\blacktriangleright$ {\bf Examples  in one  dimension.} 
\begin{enumerate} 
\item {\em Regular quadratic loss function}. $F(x)= x^2$, $U= \R$,  $\phi_{_F}(\xi,x) = (\xi-x)^2$. Any of the above distributions $P$. 
\item {\em Norm--like loss function}. $F(x)= x^a$, $a>1$, $U= (0, +\infty)$, and 
$$
\phi_{_F}(\xi,x)= \xi^a +(a-1)x^a- a\, \xi \,x^{a-1}.
$$
\begin{itemize}
\item $1<a<2$: above distributions $P$   whose support is bounded away from $0$.
\item $a=2$:  any of the above distributions $P$. 
\item $a>2$:  Above distributions $P$   with compact support in $\R_+$.
\end{itemize}

\item {\em Itakura--Saito divergence.} $F(x)= -\log (x)$, $U= (0, +\infty)$,  and 
$$
\phi_{_F}(\xi,x) = \log \Big(\frac x\xi\Big)+\frac \xi x-1.
$$
Above distributions $P$ whose support is bounded away from $0$.
\item {\em I-divergence or Kullback--Leibler divergence}.  $F(x) = x\log(x)$, $U= (0, +\infty)$, and  
$$
\phi_{_F}(\xi,x)= \xi\Big( \log\big(\frac \xi x\big)-1+\frac x\xi \Big).
$$
Above distributions $P$ whose support is bounded away from $0$.
\item {\em Logistic loss}. $F(x)= x\log x +(1-x)\log(1-x)$,  $U= (0, 1)$, and 
$$
\phi_{_F}(\xi,x)=\xi \log\Big(\frac \xi x\Big) +(1-\xi)\log\Big(\frac{1-\xi}{1-v}\Big).
$$
Above distributions $P$ with compact  support included in $(0,1)$ (i.e. bounded away from $0$ and $1$).
\item {\em Softplus loss}. $F(x)=F_a(x)= \log(1+e^{ax})/a$, $a>0$, $U= \R$, and 
$$
\phi_{_F}(\xi,x) = \frac 1a\log\Big(  \frac{1+e^{a\xi}}{1+ e^{ax}}\Big) - \frac{e^{ax}}{e^{ax} +1}(\xi-x).
$$
Any of the above distributions $P$. 
\item {\em Exponential loss}. $F_\rho(x)= e^{\rho x}$,   $\rho\!\in \R$, $U= \R$, $\phi_{F_\rho}(\xi,x) = \phi_{F_1}(\rho\, \xi,\rho\, x)$  \linebreak with $ \phi_{F_1}(\xi,x) = e^\xi -e^x-e^v(\xi- x)$. 

At least  all distributions $P$ with compact support.
 \end{enumerate} 

\medskip
$\blacktriangleright$ {\bf Examples in higher dimension.} 

\medskip

\begin{enumerate} 
 \item {\em Regular quadratic loss function}. $F(x)= |x|^2$, $x\!\in U= \R^d$.

 Any of the above distributions $P$ \textcolor{black}{with finite second moment}.
 \item {\em Mahalanobis distance}. $F(x)= x^*Sx$,  $x\!\in U= \R^d$, $S\!\in {\cal S}^{++}(d,\R)$ (symmetric and positive definite), and  
 $$
 \phi_{_F}(\xi,x)= (\xi-x)^*S(\xi-x):=|\xi-x|^2_{_S}.
 $$
 
 Any of the above distributions $P$ \textcolor{black}{with finite second moment}.
\item {\em $f$-marginal divergence (additive)}. $F(x_1,\ldots,x_d)= \sum_{i=1}^d f(x_i)$, $f$ strictly convex on $U$, is defined on $U^d$, and for every $x=(x_1,\ldots,x_d)$, $\xi=(\xi_1, \ldots,\xi_d)\!\in U^d$,
$$
\phi_{_F}(\xi,x)=  \sum_{i=1}^d \phi_{_F}(\xi_i,x_i).
$$
\color{black}Distributions to be  specified (depending essentially on the behaviour of $F$ at the boundary of the support of $P$ and at $\infty$ when this support is  not a compact included in $U$).\color{black}
  \item {\em $f$-marginal divergence (multiplicative)}. $F(x_1,\ldots,x_d)= \prod_{i=1}^d f(x_i)$, $f$ strictly convex on $U$, is defined on $U^d$, and for every $x=(x_1,\ldots,x_d)$, $\xi=(\xi_1, \ldots,\xi_d)\!\in U^d$,
  $$
  \phi_{_F}(\xi,x)=  \sum_{i=1}^d \phi_{_F}(\xi_i,x_i)\prod_{j\neq i}f(x_j).
  $$
  Same distributions $P$ as above.
  \item {\em Soft max marginal $f$-divergence}.  $F(x_1,\ldots,x_d)=F_{\lambda}(x_1,\ldots,x_d)= \frac{1}{\lambda}\log\Big(\sum_{i=1}^d e^{\lambda f(x_i)}\Big)$, $f$ strictly convex on $U$, is defined on $U^d$ and for every $x=(x_1,\ldots,x_d)$, $\xi=(\xi_1, \ldots,\xi_d)\!\in U^d$,
  $$
  \phi(\xi,x)= F_{\lambda}(\xi)-F_{\lambda}(x)-   \frac{\sum_{1\le i\le d} f'(x_i)e^{\lambda f(x_i)}(\xi_i-x_i)}{\sum_{1\le i\le d} e^{\lambda f(x_i)}}.
  $$
At least  distributions $P$ with compact support.
 \end{enumerate}  
%

\section{Proof of  Theorem~\ref{thm:ZadorBregman}: Zador's theorem  for Bregman divergence}\label{sec:proofZadorBregman}
\subsection{A first step: Zador's Theorem for  Mahalanobis divergence   and  the uniform distribution over $[0,1]^d$}
\label{subsec:Mahalanobis}

We consider in this section the case where $F$ is the squared Euclidean norm  attached to a positive definite matrix $S\! \in {\cal S}^{++}(d,\R)$ (a.k.a. Mahalanobis--Bregman divergence) so that one easily checks (what is true for any squared Euclidean norm) that 
\[
\phi_{_F}(\xi,x)=  F(\xi-x) = |\xi-x|^2_S=  (\xi-x)^*S(\xi-x).
\]
By an abuse of notation we will denote $e_r(\Gamma, P,S)$ instead of $e_r(\Gamma, P, |\cdot|^2_S)$ and $e_{r,n}(P,S)$ instead of $e_{r,n}(\Gamma, P, |\cdot|^2_S)$.

First we  easily check using the linearity of $S$ that, for any $r>0$,  and every  $A$, $B\!\in \R$, $A<B$, 
\begin{equation}\label{eq:dilatationS}
e_{r,n}\big( \mathcal{U}([A,B]^d), S\big)= (B-A)e_{r,n}\big( \mathcal{U}([0,1]^d), S\big),
\end{equation}
\color{black}where $\mathcal{U}(C)$ stands for the uniform distribution over a (non-negligible) Borel set $C$.\color{black}

\medskip
\noindent {\it Proof}. This follows from the fact the mapping {$\Gamma \mapsto \Gamma_{A,B}:= \big\{(a-A)/(B-A), \, a\!\in \Gamma\big\}$ is bijective form the set of grids of size at most $n$ into itself, and that
$S$ commutes with the dilatation $\xi\mapsto(B-A)\xi$ and 
\small
\begin{align*}
e_r \big(\Gamma, \mathcal{U}([A,B]^d\big), S)^r  & = \int_{[A,B]^d}\min_{a\in \Gamma}\big((\xi-a)^*S(\xi-a)\big)^{r/2} \frac{\mathrm{d}\xi}{(B-A)^d}\\
				   &= \int_{[0,1]^d}\min_{a\in \Gamma}\big(A+(B-A)u-a)^*S(A+(B-A)u-a)\big)^{r/2}  \mathrm{d} u\\
				   &=\int_{[0,1]^d}\min_{a\in \Gamma}\big(A+(B-A)u-a)^*S(A+(B-A)u-a)\big)^{r/2}  \mathrm{d} u\\
				   &=(B-A)^{r}  \int_{[0,1]^d}  \min_{a\in \Gamma}\big(\big(u-(a-A)/(B-A)\big)^*S\big(u-(a-A)/(B-A)\big)\big)^{r/2}  \mathrm{d} u\\
				   &= (B-A)^{r} e_r \big(\Gamma_{A,B}, \mathcal{U}([0,1]^d), S\big)^r,
\end{align*}
\normalsize
where we made the change of variable $\xi= A+(B-A)u$, $u\!\in [0,1]^d$ in the second line. \hfill $\Box$

\medskip
\noindent {\bf Remark.} The above proof obviously works when considering any squared norm as a loss function as done in regular optimal quantization theory. 

\begin{proposition}\label{prop:ZadorS}
Let $S \in \mathcal{S}^{++}(d,\mathbb{R})$ be a positive definite matrix. 

Then, 
\begin{equation}\label{eqn:zador_matrix}
    \lim_{n \rightarrow \infty} n^{1/d} e_{r,n}(\mathcal{U}([0,1]^d),S) = Q_r([0,1]^d)^{1/r}\det(S)^{1/2d}.
\end{equation}
\end{proposition}

\noindent {\em Proof.}
Note that $(\sqrt{S}u)^{\otimes 2}=u^*Su$ where $\sqrt{S}$ is the unique matrix in $\mathcal{S}^{++}(\mathrm{d},\mathbb{R})$ such that $\sqrt{S}^2=S$ (which satisfies  moreover $\sqrt{S}S=S\sqrt{S}$)
\begin{align*} 
e_{r,n}\big(\mathcal{U}([0,1]^d),S \big)^r &= \inf_{|\Gamma|\leq n} \int_{[0,1]^d} \min_{a\in\Gamma} |\sqrt{S}(x-a)^{\otimes 2}|^{r/2}  \mathrm{d} x \\
 &= \inf_{|\Gamma|\leq n} \int_{[0,1]^d} \min_{a\in\Gamma} |\sqrt{S}x-\sqrt{S}a|^{r} dx.
\end{align*}

Using the fact that $x\rightarrow\sqrt{S}x$ is a linear bijection of $\mathbb{R}^d$.
The change of variable  $x =(\sqrt{S})^{-1}u$  yields
\begin{align*}
    e_{r,n}\big (\mathcal{U}([0,1]^d),S \big)^r &= \inf_{|\Gamma|\leq n} \int_{u\in\sqrt{S}[0,1]^d} \min_{a\in\sqrt{S}\Gamma} |u-a|^r \sqrt{\det S}^{-1}du = e_{r,n}(U(\sqrt{S}[0,1]^d),|\cdot|)^r.
\end{align*}
Now
$$
\mathcal{U}(\sqrt{S}[0,1]^d)=h_S(x).\lambda_d(\mathrm{d}x) := \frac{1}{\det(\sqrt{S})}\textbf{1}_{x\in\sqrt{S}[0,1]^d}.\lambda_d(\mathrm{d}x).
$$

Since 
$$
\lambda_d(\sqrt{S}[0,1]^d) =\int_{u\in\sqrt{S}[0,1]^d} \textbf{1} \lambda_d(\mathrm{d}u) = \int_{x\in[0,1]^d}\det(\sqrt{S})\lambda_d(\mathrm{d}x)=\det(\sqrt{S})
$$
then,
$$
\limsup_{n\rightarrow+\infty} n^{r/d} e_{r,n}\big(\mathcal{U}([0,1]^d),S\big)^r = \limsup_{n\rightarrow+\infty} n^{r/d} e_{r,n}\big(U(\sqrt{S}[0,1]^d),|\cdot|\big)^r.
$$

By the original  Zador Theorem \ref{thm:reg_zador}, we get
\begin{align*}
    \lim_{n\rightarrow+\infty} n^{r/d} e_{r,n}(\mathcal{U}([0,1]^d),S)^r &=  Q_r([0,1]^d)\times\|h_S\|_{d/(d+r)} \\
    &= Q_r([0,1]^d)\times\frac{\det\sqrt{S}^{1+r/d}}{\det\sqrt{S}}.
\end{align*}
Finally
\[
    \lim_{n\rightarrow+\infty} n^{r/d} e_{r,n}([0,1]^d,S)^r =Q_r(\mathcal{U}([0,1]^d)\times(\det S)^{r/2d}
 \]
 or, equivalently, 
\[
\hskip 3,25cm     \lim_{n\rightarrow+\infty} n^{1/d} e_{r,n}([0,1]^d,S) = Q_r(\mathcal{U}([0,1]^d)^{1/r}\times(\det S)^{1/2d}. \hskip 3,25cm    \Box
\]

\subsection{Proof of Theorem~\ref{thm:ZadorBregman}}

\noindent {\it Proof.} $(a)$ In steps~1 to~7, we consider  an absolutely continuous distribution $P=h\cdot\lambda_d$  with {\em a compact support included in $U$}. Let  $C$ be a closed hypercube of $\mathbb{R}^d$ with edges parallel to the coordinate axis such that  $\mathrm{supp}(P)\subset C$. Let $L$ be  the common edge-length of $C$.  

Let $m\!\in \mathbb{N}$ and let $(C_i)_{1\leq i\leq m^d}=(C^{(m)}_i)_{1\leq i\leq m^d}$ be a covering of $m^d$   closed hypercubes of $C$ with edges parallel to the coordinate axis such that 
$$
\forall\, i\neq j,\quad \mathring{C}_i\cap \mathring{C}_j= \varnothing \quad \mbox{ and }\quad \bigcup_{i=1,\ldots,m^d} C_i = C.
$$ 
We denote by  $c_i$ the midpoint of $C_i$. Note that all   these small hypercubes  are translated in $\mathbb{R}^d$ from $[0, \frac{L}{m}]^d$.  Hence, their common diameter (for the canonical Euclidean norm) is $\frac{\sqrt{d}L}{m}$.
At some places   we will implicitly consider   that  these hypercubes are half-open so that the family \textcolor{black}{$(C_i)_{1\le i\le m^d}$ makes} up a true  {\em partition} of $C$ with of course  no impact on the proofs.

As ${\rm supp}(P) $ is compact it is clear that 
$$
\ve_0:=\tfrac 13d\big({\rm supp}(P), U^c\big)>0,
$$ 
where $d\big({\rm supp}(P), U^c\big):=\inf_{x\in {\rm supp}(P)}d(x,U^c)$, so that 
\[
{\rm supp}(P) \subset U_{2\ve_0}\subset \bar U_{2\ve_0}.
\]
Let us define
\[
I_m= \big\{i\!\in \{1,\ldots,m^d\}: C_i\cap \bar U_{2\ve_0}   \neq \varnothing\big\}.
\]
As the diameter of hypercubes $C_i$ is $\frac{\sqrt{d}L}{m}$, it is clear that for large enough $m$, namely 
\begin{equation}\label{eq:m0}
m\ge m_0:= \left\lfloor \frac{\sqrt{d}L}{\ve_0}\right\rfloor+1,
\end{equation}
one has 
\[
\textcolor{black}{
C\cap \bar U_{2\ve_0}\subset \bigcup_{i\in I_m}C_i \subset     C \cap U_{2\ve_0-\frac{\sqrt{d}L}{m}}\subset C \cap U_{\ve_0}\subset C \cap  \bar U_{\ve_0}
}
\]
since, for $m>\frac{2\sqrt{d}L}{\ve_0}$ i.e. $\frac{\sqrt{d}L}{m}<\ve_0$. Finally,
\[
{\rm supp}(P) \subset \bigcup_{i\in I_m}C_i \subset C \cap  \bar  U_{\ve_0}\subset U.
\] 

In particular the function $F$ is well-defined on every small hypercube $C_i$, $i\!\in I_m$. At this stage we may define for every integer $m\ge1$, 
\begin{equation}\label{eq:epsilonm}
\ve_m = \max_{i\in I_m} w \Big (\nabla^2 F,C_i,\frac{\sqrt{d}L}{m}\big)\le w \Big (\nabla^2 F,\bar  U_{\ve_0}\cap C,\frac{\sqrt{d}L}{m}\Big),
\end{equation}
\textcolor{black}{where $w(f,A,\delta)$ denotes the continuity modulus with amplitude $\delta$ of a function $f:\R^d\to \R^q$ restricted to a subset $A\subset \R^d$, $\R^d$ being equipped with the canonical Euclidean norm}. As $C \cap \bar  U_{\ve_0}$ is compact (as a closed subset of the compact set $C$) and $\nabla^2F$ is continuous on $U$ hence uniformly continuous on $C \cap \bar  U_{\ve_0}$, we have
\[
\forall\, m\ge 1, \quad\ve_m<+\infty \quad \mbox{ and }\quad \lim_{m\to +\infty}\ve_m = 0.
\]
\noindent  {\sc Step 1} ({\em Upper Bound for the proxy $P_m$ of $P$}). As a preliminary, note that $F$ being continuously twice differentiable, a second order Taylor expansion yields, for every $x, a\!\in U$
$$
\phi_{_F}(\xi,a) = F(\xi) - F(a) - \langle \nabla F(a)\,|\,\xi-a\rangle = \left ( \int_0^1 (1-u) \nabla^2F(a+u(\xi-a)) \mathrm{d}u \right ) \cdot (\xi-a)^{\otimes2}.
$$

Let $\Gamma \subset U $ be a (finite) quantizer such that for every $i\! \in I_m$, $|\Gamma\cap C_i|>0  $.  For every $\xi\!\in C_i$ we consider  the {\em local nearest neighbour} $a(\xi)\in C_i\cap \Gamma$  defined by 
$$
    a(\xi) \in  {\rm argmin}_{a\in\Gamma\cap \,C_i} \big(F(\xi) - F(a) - \langle \nabla F(a) \,|\, \xi-a \rangle\big). 
$$

This local assignment rule is clearly  sub-optimal  since we restrict our search for the nearest neighbour of $\xi \!\in C_i$ to $ C_i\cap \Gamma$.
But we may reasonably hope that it will be enough to get a sharp upper bound, at least when $m$ is large enough. 

As $C_i$ is convex with diameter $\frac{\sqrt{d}L}{m}$, if $\xi,a\!\in C_i$ then  for every $ u \in [0,1]$, $a+u(\xi-a) \in C_i$ and 
\[
\vertiii{ \nabla^2F(a+u(\xi-a)) - \nabla^2 F(c_i)} \le  w\Big (\nabla^2 F, C_i, \frac{\sqrt{d}L}{m}\Big).
\]  
In particular this implies by the ``left'' triangle inequality that
\[
\vertiii{ \nabla^2F(a+u(\xi-a))} \le\vertiii{ \nabla^2 F(c_i) }+ w\Big (\nabla^2 F, C_i, \frac{\sqrt{d}L}{m}\Big).
\]
For a positive definite symmetric matrix $S$ (such is the case of   $\nabla^2F(a+u(x-a))$ and $\nabla^2 F(c_i)$), $\vertiii{S}= \sup_{u:|u|=1} u^*Su$. Hence, one easily checks that  the  above inequality can be rewritten \color{black}in terms of ordering of symmetric matrices~(\footnote{$S\le T$ if $T-S \!\in {\cal S}^+(d, \R)$.}) \color{black} as 
\begin{align*}
    \nabla^2F(a+u(\xi-a)) &\leq \nabla^2 F(c_i) + w\Big (\nabla^2 F, C_i, \frac{\sqrt{d}L}{m}\Big)I_d
    \\
    &= \nabla^2F(c_i) + \ve_m I_d.
\end{align*}

Then, for every $\xi,\, a \in C_i$,
\begin{align}
    \nonumber \int_0^1(1-u)\nabla^2F(a+u(\xi-a))\mathrm{d}u &\leq  \int_{0}^1 (1-u)(\nabla^2F(c_i)+\ve_m I_d)\textbf{1}_{\{a+u(\xi-a)\in C_i\}} \mathrm{d}u \\
  \label{eq:majorBregdiv} &= \frac{1}{2}\sum_{i\in I_m} (\nabla^2 F(c_i) + \ve_m I_d),
\end{align}
where $I_d$ denotes the identity  of the space $\R^d$.
We set 
\begin{align}
\label{eq:nabla2Fm}    \nabla^2 F_m(x) &= 
\sum_{i\in I_m}(\nabla^2 F(c_i) + \ve_m I_d) \textbf{1}_{\{x\in C_i\}},\\
\label{eq:Pmhm}     P_m &= \sum_{i\in I_m} P(C_i)\mathcal{U}(C_i) \quad\text{ and } \quad h_m = \sum_{i\in I_m} \frac{P_m(C_i)}{\lambda_d(C_i)}\textbf{1}_{C_i}= \Big(\frac mL\Big)^d\sum_{i\in I_m}P_m(C_i)\textbf{1}_{C_i},
\end{align}
where $\mathcal{U}(C_i) $ denotes the  uniform distribution over the hypercube $C_i$.
As a consequence we obtain the upper-bound
\begin{equation}\label{eq:magotphi1}
\forall\, {i\in I_m}, \; \forall\, \xi,\, a \!\in C_i, \quad \phi_{_F}(\xi,a)\le\textcolor{black}{ \tfrac 12}\nabla^2 F_m(\xi,a) (\xi-a)^{\otimes 2}.
\end{equation}

Then, it follows from~\eqref{eq:magotphi1}  and the above definition of $\nabla^2F_m$ that, for any quantizer $\Gamma\subset U$, 
\begin{align}
   \nonumber  e_{r}(\Gamma, P_m,\phi_{_F})^r &\leq 2^{-\frac r2} \int_{\cup_{i\in I_m}C_i} \min_{a\in \Gamma} \big(\nabla^2F_m(\xi)(\xi-a)^{\otimes2}\big)^\frac{r}{2} P_m(\mathrm{d}\xi)\\
   \nonumber  &=2^{-\frac r2}  \sum_{i\in I_m} P(C_i)  \int_{C_i} \min_{a\in \Gamma} \big(\nabla^2F_m(\xi)(\xi-a)^{\otimes2}\big)^\frac{r}{2} \frac{\mathrm{d}\xi}{\lambda_d(C_i)}\\
   \label{eq:majorer} &\leq 2^{-\frac r2} \Big(\frac{m}{L}\Big)^d \sum_{i\in I_m} P(C_i) \int_{C_i} \min_{a\in\Gamma\cap C_i}\big(\nabla^2F_m(\xi)(\xi-a)^{\otimes2}\big)^{\frac{r}{2}} \mathrm{d}\xi,
\end{align}
where we used in the last line that  for every $i\!\in I_m$, $\Gamma\cap C_i\subset \Gamma$ and  $\lambda_d(C_i)\leq (L/m)^d$. Indeed this is the mathematical form of  our local (suboptimal) assignment rule .

Let $n_i= |\Gamma\cap C_i |$ and let $\Gamma^{o,i}_{n_i}= \big\{\frac{m}{L}(a-c_i), \; a\!\in \Gamma\cap C_i\big\}\subset [-\frac 12,\frac12]^d$.   
The change of variables $\xi := c_i + \frac{L u}{m}$ yields :
\begin{align*}
    \int_{C_i} \min_{a\in\Gamma\cap C_i} (\nabla^2F_m(\xi)&(\xi-a)^{\otimes2})^{\frac{r}{2}} \lambda_d(\mathrm{d}\xi) \\
    &= L^d m^{-d}\int_{[-\frac{1}{2},\frac{1}{2}]^d} \min_{a'\in\Gamma^{o,i}_{n_i}}\Big ( \nabla^2 F_m(c_i)\Big (\Big (\frac{L u}{m}+c_i\Big )-\Big (\frac{L a'}{m}+c_i\Big )\Big)^{\otimes2} \Big )^{\frac{r}{2}}\lambda_d(\mathrm{d}u) \\
    & = L^{d+r} m^{-(d+r)}\int_{[-\frac{1}{2},\frac{1}{2}]^d} \min_{a'\in\Gamma^{o,i}_{n_i}} (\nabla^2F_m(c_i)(u-a')^{\otimes2})^{\frac{r}{2}} \lambda_d(\mathrm{d}u).
\end{align*}

Plugging this identity into~\eqref{eq:majorer} yields 
\begin{align*}
     e_{r}(\Gamma, P_m,\phi_{_F})^r & \le 2^{-\frac r2}  L^rm^{-r} \sum_{i\in I_m}\,\textcolor{black}{P(C_i)} \int_{[-\frac{1}{2},\frac{1}{2}]^d} \min_{a'\in\Gamma^{o,i}_{n_i}} (\nabla^2F_m(c_i)(u-a')^{\otimes2})^{\frac{r}{2}} \lambda_d(\mathrm{d}u).
\end{align*}
Hence, by Proposition~\ref{prop:ZadorS} applied to the hypercube  $[-\frac{1}{2},\frac{1}{2}]^d$ and the fixed distortion matrix $S=\nabla^2 F_m(c_i)$, we get:  
\[
     e_{r,n}(\Gamma, P_m,\nabla^2F)^r \leq 2^{-\frac r2}  L^rm^{-r} \sum_{i\in I_m} n_i^{-\frac{r}{d}} (1 + \eta_i(n_i)) \det(\nabla^2 F_m(c_i))^{\frac{r}{2d}} P(C_i) Q_r([0,1]^d),
\]
where $\displaystyle \lim_{n\rightarrow +\infty} \eta_i(n)=0$ so that
$$
    e_{r,n}(P_m,\nabla^2F)^r \leq 2^{-\frac r2}  Q_r([0,1]^d) L^r m^{-r} \max_{i\in I} (1 + \eta_i(n_i)) \sum_{i\in I_m} n_i^{-\frac{r}{d}} \det(\nabla^2 F_m(c_i))^{\frac{r}{2d}} P(C_i).
$$

To specify the quantizer $\Gamma$ and in particular the integers $n_i= |\Gamma\cap C_i|$, we rely on the reverse H\"older inequality: let $x_i \!\in (0,+\infty) $, $i\in I_m$ such that $\sum_{i\in I_m}x_i=1$ and let $y_i$, $i\!\in I_m$ be positive real  numbers. The reverse H\"older inequality applied to the conjugate exponents $p=-\frac{d}{r}<0$ and $q=\frac{d}{d+r}\!\in (0,1)$ implies that  
\begin{equation}
    \sum_{i\in I_m} x_i^{-\frac{r}{d}}y_i \geq \left ( \sum_{i\in I_m} x_i \right )^{-\frac{r}{d}} \left ( \sum_{i\in I_m} y_i^{\frac{d}{d+r}} \right )^{1+\frac{r}{d}} = \left (\sum_{i\in I_m} y_i^{\frac{d}{d+r}} \right )^{1+\frac{r}{d}}
\end{equation}
with equality if and only if $x_i = \frac{y_i^{d/d+r}}{\sum_{j\in I_m}y_j^{d/d+r}}>0$.

\smallskip
We apply this result to $y_i = \det(\nabla^2 F_m)^{\frac{r}{2d}}P(C_i)>0$ and we set $n_i = \lfloor n x_i\rfloor \rightarrow\infty$ as $n\rightarrow +\infty$ so that
\[
n^{r/d} \sum_{i\in I_m} \lfloor nx_i\rfloor^{-\frac{r}{d}}\det(\nabla^2 F_m)^{\frac{r}{2d}}P(C_i)\to\left( \sum_{i\in I_m}\big(\det(\nabla^2 F_m)^{\frac{r}{2d}}P(C_i)\big)^{\frac{d}{d+r}}\right)^{1+\frac{r}{d}}\quad \mbox{as}\quad n\to+\infty.
\]

\color{black}Hence
\begin{align*}
    \limsup_{n\rightarrow+\infty} n^{r/d} e_{r,n}(P_m,\nabla^2 F_m)^r & \leq  2^{-\frac r2} Q_r([0,1]^d)\textcolor{black}{L^r }m^{-r} \left (\sum_{i\in I_m}\det(\nabla^2 F_m(c_i))^{\frac{r}{2d}\cdot\frac{d}{d+r}}P(C_i)^{\frac{d}{d+r}}\right)^{1+\frac{r}{d}} \\
    &=2^{-\frac r2}  Q_r([0,1]^d)\textcolor{black}{L^r } \left (\sum_{i\in I_m}\det(\nabla^2 F_m(c_i))^{\frac{r}{2(d+r)}}m^{-\frac{rd}{r+d}}P(C_i)^{\frac{d}{d+r}}\right)^{1+\frac{r}{d}}\!\!\!.
\end{align*}
On the other hand, one has, by the definition of $h_m$,
\begin{align*}
    \int_{\R^d} \det(\nabla^2 F_m)^{\frac{r}{2(d+r)}}(x)h_m^{\frac{d}{d+r}}(x) \lambda_d(\mathrm{d}x) &= \sum_{i\in I_m} \Big(\frac{m}{L}\Big)^{\frac{d^2}{d+r}}P(C_i)^{\frac{d}{d+r}}\det(\nabla^2 F_m(c_i))^{\frac{r}{2(d+r)}} \Big(\frac{L}{m}\Big)^{d} \\
    &= \sum_{i\in I_m}\Big(\frac m L\Big)^{\frac{d^2}{d+r}-d }P(C_i)^{\frac{d}{d+r}}\det(\nabla^2 F_m(c_i))^{\frac{r}{2(d+r)}}\\
    &=L^{\frac{rd}{r+d}}  \sum_{i\in I_m} m^{-\frac{rd}{d+r}}P(C_i)^{\frac{d}{d+r}}\det(\nabla^2 F_m(c_i))^{\frac{r}{2(d+r)}}
\end{align*}
so that
\[
\left(\sum_{i\in I_m} m^{-\frac{rd}{d+r}}P(C_i)^{\frac{d}{d+r}}\det(\nabla^2 F_m(c_i))^{\frac{r}{2(d+r)}}\right)^{1+\frac rd}= L^{-r}  \|\det(\nabla^2 F_m)^{\frac{r}{2d}}h_m\|_{L^{\frac{d}{d+r}}(\lambda_d)} .
 \]
 Consequently
\begin{equation}\label{eq:limsupPm}
\limsup_{n\rightarrow+\infty} n^{r/d} e_{r,n}(P_m,\nabla^2 F_m)^r \leq 2^{-\frac r2}  Q_r([0,1]^d) \|\det(\nabla^2 F_m)^{\frac{r}{2d}}h_m\|_{L^{\frac{d}{d+r}}(\lambda_d)}.
\end{equation}
\color{black}

\noindent  {\sc Step 2}
({\em Lower bound for Bregman~I: preliminaries}). As $C \cap \bar U_{\ve_0}$ is a compact subset of $U$   as well as the unit (Euclidean) sphere of $\R^d$, temporarily denoted $ S(0,1)$,  the mapping
$ S(0,1)\times (C \cap \bar U_{\ve_0})\owns (u,\xi)\mapsto u^*\nabla^2 F(\xi)u\!\in \mathbb{R}_+$ is continuous and, in fact,  always positive since  $\nabla^2F(\xi)\!\in {\cal S}^{++}(d,\mathbb{R})$ for every $\xi\!\in U$. Consequently there exists $\kappa_{\min}>0$ such that

\[
\forall\, u\!\in \mathbb{R}^d, \; \forall\, x\!\in C \cap \bar U_{\ve_0}, \quad u^*\nabla^2F(x)u \ge \kappa_{\min} |u|^2
\]  
or, equivalently,
\begin{equation}\label{eq:ellipticity}
\forall\, x\!\in C \cap \bar U_{\ve_0}, \quad  \nabla^2F(x)\ge \kappa_{\min}I_d.
\end{equation}

%

Moreover, as $\nabla^2F$ is continuous on $U$, it is uniformly continuous on $C \cap \bar U_{\ve_0}$. Hence 
\[
\lim_{\delta\to 0} w\left (\nabla^2 F, C \cap \bar U_{\ve_0},\delta\right )=0.
\]

Let us consider again the tessellation $(C_i)_{\{i=1,\ldots,m\}}$ of an hypercube $C$ with edge-length $L>0$ that contains $U$ and the associated notations ($C_i$ are hypercubes centered at $c_i$ with  edge-length $\frac{L}{m}$, $I_m$, and diameter $\frac{\sqrt{d}L}{m}$,  etc) from Step~1. 
In particular,  for $m$ large enough, say $m\ge m_1$ (we assume that $m_1\ge m_0$), 
\[
\max_{i\in I_m}  w\left ( \nabla^2 F, C_i,\frac{\sqrt{d}L}{m}\right )\le w\left ( \nabla^2 F, C \cap \bar U_{\ve_0},\frac{\sqrt{d}L}{m}\right )= \ve_m \le \frac{\kappa_{\min}}{2}.
\]
Consequently,  one has for the pre-order on ${\cal S}^+(d, \R)$, 
\begin{align*}
 \forall\, i\!\in I_m, \;\forall\, \xi \!\in C_i, \quad    \nabla^2 F(\xi) &\ge   \nabla^2 F(c_i) -  w\left ( \nabla^2 F, C_i, \frac{\sqrt{d}L}{m} \right )  I_d\\
 &    \ge \nabla^2 F(c_i) - \ve_m  I_d \geq  \big(\kappa_{\min}- \frac{\kappa_{\min}}{2}\big)I_d= \frac{\kappa_{\min}}{2}I_d >0.
 \end{align*}

Consequently, for every $i\!\in I_m$, every $\xi$, $a\!\in C_i$ 
\begin{align}
\nonumber \phi_{_F}(\xi,a) &= \int_0^1(1-u)  \nabla^2F(a+u(\xi-a))\mathrm{d}u (\xi-a)^{\otimes 2} \\
 \nonumber&\ge \frac 12  \big(  \nabla^2 F(c_i) - \ve_m  I_d\big) (\xi-a)^{\otimes 2} \\
 \label{eq:minorphi}& = \tfrac 12 \big(\nabla^2F_m(\xi) - 2\ve_m I_d\big)(\xi-a)^{\otimes 2}.
 \end{align}

%
%
\noindent  {\sc Step 3} ({\em Lower bound for Bregman~II: Firewall lemma}). 
%
%
The firewall lemma proves that one may find  finitely many points of the boundary of an hypercube so that any interior point far enough from the boundary is closer to this set of points than to any point outside the hypercube. 

This will be extensively applied to the small hyper-cubes $C_i$ of a tessellation to establish the lower bound for the Bregman quantization error. 

\begin{proposition}[Firewall Lemma] \label{lem:FirewallLemma}Let $C_i\subset C \cap\bar U_{\ve_0}$, $i\!\in I_m$,  be a closed hypercube with edges parallel to the coordinate  axis with length $L/m>0$ and  center $c_i$.
Let $\varpi \!\in (0, L/2m]$ and let $C_{i,\varpi} $ be the  hypercube with edge-length $L/m-2\varpi$ obtained as the image of $C_i$ by the contraction centered at $c_i$ with ratio $1-\varpi$  (see Figure~\ref{fig:firewall}).
Then there exists a finite set $\gamma_i = \gamma_i^{(\varpi)}\subset \partial C_{i,\varpi}$ (boundary of $C_{i,\varpi}$) such that 
$$
\forall \xi \in C_{i,\varpi}, \quad \min_{a\in\gamma_i} \phi_{_F}(\xi,a) \leq \min_{x\in C\setminus C_{i} } \phi_{_F}(\xi,x).
$$
Moreover the cardinality of the sets $\gamma_i$, $i\!\in I_m$,  only depends on the operator norm \linebreak $\displaystyle \vertiii{\nabla^2 F}_{C \cap \bar U_{\ve_0}} := \sup_{\xi\in C \cap \bar U_{\ve_0}}\vertiii{\nabla^2 F(\xi)}$, $\varpi$, $L$, $d$ and the uniform continuity modulus  $w(\nabla^2F, C \cap \bar U_{\ve_0},\cdot)$ on the compact  $C \cap \bar U_{\ve_0}$.
\end{proposition}

\noindent {\bf Remarks.} $\bullet$ This Proposition    is probably    the most demanding  step of the proof of Zador's Theorem for Bregman divergences proof in the sense that it significantly differs from its counterpart in~\cite[Section 6] {GrafL2000} (this reference provides the first fully rigorous proof of Zador's Theorem in the ``classical'' setting where  the loss function is the power of  a norm). 
This is due to the fact that, by contrast with this classical setting,  Bregman divergences $\phi_{_F}$ are never isotropic except precisely when $F$ is the squared canonical Euclidean norm (up to an affine function).  Let us make things more precise. 

Assume that $F$ is twice differentiable on $U=\R^d$ and satisfies 
\[
\forall\, \xi\, x\!\in \R^d,\quad \phi_{_F}(\xi, x) = \varphi(\xi-x),
\]
where $\varphi:\R^d \to \R_+$ is differentiable. Then, as 
\[
\forall\, \xi,\, x\!\in \R^d,\quad \partial_x\phi_{_F}(\xi,x) = -\nabla F(x) +\nabla F(x) -\nabla^2F(x)(\xi-x)=-\nabla^2F(x)(\xi-x)
\]
the above equality implies
\[
\forall\, \xi,\, x\!\in \R^d,\quad\nabla^2F(x) (\xi-x)= \nabla \varphi(\xi-x)
\]
or, equivalently,
\[
\forall\, x,\,h\!\in \R^d,\quad \nabla^2F(x)h = \nabla\varphi(h).
\]
This in turn implies that
\[ 
\forall\, x,\,y,\,h\!\in \R^d,\quad \nabla^2F(x)h = \nabla^2F(y)h 
\]
i.e. $\nabla^2F(x)= \nabla^2F(y)$.  This implies that $\nabla^2 F$ is constant $i.e.$ that there exists $S\!\in {\cal S}^{++}(d,\R)$, $a\!\in \R^d$ and $b\!\in \R$ such that 
\[
F(x) = |x|^2_S +\langle a,x\rangle +b
\]
since $F$ is assumed  to be strictly convex (and of course the affine part plays no role).

\medskip
\noindent $\bullet$ We also need a kind of firewall lemma in the next chapter devoted to the ``companion'' empirical measure results related to our Zador like theorem. However it is not powerful enough to help us in the proof of  the lower bound in the theorem due to the control of the size of the walls across all the hypercubes $C_i$ in a non-isotropic setting as emphasized above. 

\bigskip
\noindent {\em Proof.} Due to its technicality, the proof is postponed to an appendix.
\hfill $\Box$ 

\medskip
\noindent {\sc Step 4} {(\em Lower bound for Bregman~III: The case of $P_m$)}. 
Now, with the above Firewall lemma (see Proposition~\ref{lem:FirewallLemma}), we can control the lower bound of the quantization error in each hypercube $C_i$ (by introducing the $\varpi$-interior $C_{i,\varpi}$ of $C_i$ with $\varpi \!\in (0, \frac{L}{2m})$).  Let $\Gamma\subset U$ be a quantizer. One has by the elementary Bayes formula
    \color{black} 
    \begin{align*}
    e_{r,n}(\Gamma, &P_m,\phi_{_F})^r = 2^{-\frac r2}\Big(\frac mL\Big)^d \sum_{i\in I_m} P(C_i) \int_{C_i} \min_{a\in\Gamma} \phi_{_F}(\xi,a)^{  r/2} \mathrm{d}\xi \\
&\ge 2^{-\frac r2}\Big(\frac mL\Big)^d \sum_{i\in I_m} P(C_i) \int_{C_i} \min_{a\in\Gamma\cup\gamma_i} \phi_{_F}(\xi,a)^{  r/2} \mathrm{d}\xi\\
    &\ge2^{-\frac r2}\Big(\frac mL\Big)^d \sum_{i\in I_m}  P(C_i) \int_{C_{i}} \min_{a\in\Gamma\cup \gamma_i}  \big((\nabla^2F_m(c_i)-\ve_m I_d)(\xi-a)^{\otimes2}\big)^{r/2} \mathrm{d}\xi\\
         &\geq 2^{-\frac r2}\Big(\frac mL\Big)^d \sum_{i\in I_m}  P(C_i) \int_{C_{i,\varpi}} \min_{a\in\Gamma\cup\gamma_i}  \big((\nabla^2F_m(c_i)-\ve_m I_d)(\xi-a)^{\otimes2}\big)^{r/2} \mathrm{d}\xi\\
         &= 2^{-\frac r2}\Big(\frac mL\Big)^d \sum_{i\in I_m}  P(C_i) \int_{C_{i,\varpi}} \min_{a\in(\Gamma\cap \mathring{C}_i)\cup\gamma_i}  \big((\nabla^2F_m(c_i)-\ve_mI_d)(\xi-a)^{\otimes2}\big)^{r/2} \mathrm{d}\xi
\end{align*}
\color{black}
where we used the  above firewall lemma (Propsosition~\ref{lem:FirewallLemma}) in the last line to restrict the set over which is taken the minimum. Consequently
\begin{align*}
    e_{r,n}(\Gamma, P_m,\phi_{_F})^r &\ge 2^{-\frac r2}\Big(\frac mL\Big)^d\sum_{i\in I_m} P(C_i) e_{n_i +\nu_i}\big(\mathcal{U}(C_{i,\varpi}),(\nabla^2F_m(c_i)-\ve_mI_d)\big)^r
    \end{align*}
where \textcolor{black}{$n_i = |\Gamma\cap C_i|$ and $\nu_i  |\gamma_i|$ can be taken constant equal to $\nu$} (all the $\gamma_i$ can be chosen in such a way to have the same size according to the Firewall Lemma). 

\color{black}{As the right hand side does not depend on $\Gamma$, this yields
\begin{align*}
    e_{r,n}(P_m,\phi_{_F})^r &\ge 2^{-\frac r2}\Big(\frac mL\Big)^d\sum_{i\in I_m} P(C_i) e_{n_i +\nu}\Big(\mathcal{U}(C_{i,\varpi}), (\nabla^2F_m(c_i)-\ve_mI_d)\Big)^r\\
    & =  2^{-\frac r2}\Big(\frac mL\Big)^d\sum_{i\in I_m} P(C_i) \bigg(\frac Lm-\varpi\bigg)^{r+d}\!\!\!e_{n_i +\nu}\Big(\mathcal{U}\Big(\big[-\tfrac 12, \tfrac 12\big]^d\Big), (\nabla^2F_m(c_i)-\ve_mI_d)\Big)^r.
    \end{align*}
\color{black}    
     Up to an extraction  (still denoted  $n$ for convenience)  we may assume that all the sequences $\frac{n_i+\nu}{n}\to v_i\!\in [0,1]$ as $n\to +\infty$ where the $v_i$, $i\!\in I_m$  satisfy  $\sum_{i\in I_m} v_i \le 1$. Moreover,  by the $\limsup_n$ result for $n^{r/d} e_{r,n}(P_m,\phi_{_F})^r$ established in Step~1, we may also    assume that 
     \begin{align*}
 \Lambda_m &:=     \liminf_{n\to+\infty}n^{\frac rd}\left (2^{-\frac r2}\Big(\frac mL\Big)^d \Bigg(\frac Lm-\varpi\Bigg)^{r+d}\sum_{i\in I_m} P(C_i)e_{n_i +\nu}\Big(\mathcal{U}\Big(\big[-\tfrac 12, \tfrac 12\big]^d\Big), (\nabla^2F_m(c_i)-\ve_mI_d)\Big)^r\right)\\
&<+\infty.
     \end{align*}
  Then, using that  $ \lim_n \Big(\frac{n}{n_i+\nu} \Big)^{\frac rd} = v_i^{-\frac rd}\!\in (0, +\infty]$ combined with  the classical properties of $\liminf_n$, it follows from  Proposition~\ref{lem:FirewallLemma} applied with $S= \nabla^2F_m(c_i)-\ve_mI_d$, $i\!\in I_m$ that 
   \[
 \Lambda_m\ge 2^{-\frac r2}Q_r([0,1]^d\Big( \frac mL\Big)^d \Bigg(\frac Lm-\varpi\Bigg)^{r+d}\sum_{i\in I_m} P(C_i)v_i^{-\frac rd}\det \big( \nabla^2F_m(c_i)-\ve_mI_d\big)^{\frac{r}{2d}}.
     \]
Note that if  any of the $v_i$ is $0$ then $\Lambda_m=+\infty$ which yields a contradiction. Hence $v_i>0$ for every $i\!\in I_m$. This holds for any $\varpi $ small enough so that letting $\varpi \to 0$ yields
 \[
 \Lambda_m\ge 2^{-\frac r2}Q_r([0,1]^d)\bigg( \frac Lm\bigg)^r \sum_{i\in I_m} P(C_i)v_i^{-\frac rd}\det \Big(\nabla^2F_m(c_i)-\ve_mI_d\Big)^{\frac{r}{2d}}.
  \]

At this stage we may apply the reverse H\"older inequality with conjugate exponents $p= -\frac dr$ and $q = \frac{d}{d+r}$ which  shows that
\begin{align*}
\Lambda _m &\ge 2^{-\frac r2} Q_r([0,1]^d)\bigg( \frac Lm\bigg)^r  \left(\sum_{i\in I_m} P(C_i)^{\frac{d}{d+r}} \det \Big( \nabla^2F_m(c_i)-\ve_mI_d\Big)^{\frac{r}{2(d+r)} } \right)^{\frac{d+r}{d}}       \underbrace{\Big(\sum_{i\in I_m} v_i\Big)^{-\frac r d}}_{\geq 1}\\
& \ge 2^{-\frac r2} Q_r([0,1]^d)\bigg( \frac Lm\bigg)^r  \left(\sum_{i\in I_m} P(C_i)^{\frac{d}{d+r} }\det \Big( \nabla^2F_m(c_i)-\ve_mI_d\Big)^{\frac{r}{2(d+r)}}  \right)^{\frac{d+r}{d}}.
\end{align*}

As $\displaystyle \liminf_{n\to+\infty} n^{\frac rd}  e_{r,n}(P_m,\phi_{_F})^r \ge \Lambda_m$ the above right hand side of the inequality is also a lower bound for  $\displaystyle \liminf_{n\to+\infty} n^{\frac rd}  e_{r,n}(P_m,\phi_{_F})^r $.

\smallskip
Finally note that  
\begin{align*}
\bigg( \frac Lm\bigg)^r  \Bigg(\sum_{i\in I_m} P(C_i)^{\frac{d}{d+r} } &\det \!\Big( \nabla^2F_m(c_i)-\ve_mI_d\Big)^{\frac{r}{2(d+r)}}  \Bigg)^{\frac{d+r}{d}}\\
&= \Bigg(\sum_{i\in I_m} \Bigg(\frac{P(C_i)}{\lambda_d(C_i)}\Big)^{\frac{d}{d+r} }\det \Big( \nabla^2F_m(c_i)-\ve_mI_d\Big)^{\frac{r}{2(d+r)}} \lambda(C_i) \Bigg)^{\frac{d+r}{d}}\\
& = \Big\| h_m\det(\nabla^2F_m -\ve_mI_d)\Big\|_{L^{\frac{d}{d+r}}(\lambda_d)}
\end{align*}
so that
 \begin{equation}\label{eq:minorP_m}
 \liminf_{n\to +\infty} n^{\frac rd}  e_{r,n}(P_m,\phi_{_F})^r \ge2^{-\frac r2}  Q_r([0,1]^d) \Big\| h_m\det\Big(\nabla^2F_m -\ve_mI_d\Big)\Big\|_{L^{\frac{d}{d+r}}(\lambda_d)}.
 \end{equation}  
    
%
\noindent {\sc Step 5} ({\em Toward  upper and lower bounds: $P$ versus $P_m$}). 
We need to control the distortions of order $r$ induced by $P= h\cdot \lambda_d$ and that of the approximation $P_m$ of $P$ investigated in the previous steps. To be more precise,  we aim at controlling the following normalized error
\[
    n^{r/d} |e_{r,n}(P,\phi_{_F})^r - e_{r,n}(P_m,\phi_{_F})^r|.
\]
 
Let $\ve \in (0,1)$ and $n\geq \frac{1}{\ve}\vee \frac{1}{1-\ve}$. Set $n_1(\ve)=\lfloor(1-\ve)n\rfloor \geq 1$, and  $ n_2(\ve) = \lfloor(\ve n)^{1/d}\rfloor^d\geq 1 $  so that $n_1+n_2\le n$.

We consider a covering $(C_i)_{i=1, \ldots, n_2(\ve)}$ of $C$ by $n_2(\ve)^{\frac{1}{d}}$ closed hypercubes with edges parallel to the coordinate axis and common length $\frac{L}{n_2(\ve)^{1/d}}$. 
We denote by $\Gamma_{\ve,n}^o\subset C^{(n_2(\ve)^{1/d})}$ the set of their centers. 

One has, for large enough $n$ such that $n_2(\ve)^{1/d}\ge m_0$ (see~\eqref{eq:m0}),
$$
\forall \xi \!\in C,
\quad \min_{a\in\Gamma^o_{\ve,n}} |a-\xi| \leq \frac{\sqrt{d}L}{2n_2(\ve)^{1/d}}
$$
so that
\begin{equation} \label{eq:1}
    n^{\frac{r}{d}}\max_{\xi \in C} \min_{a\in\Gamma_{\ve,n}^o} |\xi-a|^r\leq\frac{d^{\frac{r}{2}}L^r n^{\frac{r}{d}}}{2^r\lfloor(n\ve)^{\frac{1}{d}}\rfloor^r} \leq \frac{d^{\frac{r}{2}}L^r}{2^r\ve^{\frac{r}{d}}}.
\end{equation}

Let $\Gamma_{n,\ve}$ be any  quantizer with values in $U$ of size $n_1(\ve)$ and set $\widetilde \Gamma_n = \Gamma_{n,\ve}\cup \Gamma_{\ve,n}^o $ which has a size at most $n$ by construction.
Thanks to the regularity of $F$ and  the fact that $h$ and $h_m$ are null outside $\bar U_{\ve_0}$, we have  
\begin{align}
\nonumber     n^{\frac{r}{d}}\big|e_{r,n}(\widetilde \Gamma_n&\cap U,P_m,\phi_{_F})^r -e_{r,n}(\widetilde \Gamma_n\cap U, P,\phi_{_F})^r\big| \\
\nonumber      &=    n^{\frac{r}{d}}\left|\int \min_{a\in \widetilde \Gamma_n\cap U}\big(F(\xi)-F(a)-\langle \nabla F(a)\,|\, \xi-a\rangle\big)^{\frac r2}\big( h(\xi)-h_m(\xi)  \big)\lambda_d(\mathrm{d}\xi)  \right|  \\
\nonumber      &\le n^{\frac rd} 2^{-\frac r2} \vertiii{\nabla^2 F}_{C \cap \bar U_{\ve_0}}^{\frac r2} \int_U \min_{a\in \widetilde \Gamma_n\cap U} |\xi-a|^r |h(\xi)-h_m(\xi)| \lambda_d(\mathrm{d}\xi)  \\
  \nonumber      & \leq 2^{-\frac r2}\vertiii{\nabla^2 F}_{C \cap \bar U_{\ve_0}}^{\frac r2}   n^{\frac rd}\max_{\xi \in C} \min_{a\in\widetilde \Gamma_n}|\xi-a|^r\|h_m-h\|_1 \\
 \label{eq:comapardistir}   &\leq \underbrace{\vertiii{\nabla^2 F}_{C \cap \bar U_{\ve_0}} 2^{-\frac r2} \frac{d^{\frac r2} L^r}{2^r}}_{C_{d,r,L,\nabla^2F}}\ve^{-\frac{r}{d}}\|h_m-h\|_1.
\end{align}

\textcolor{black}{$\rhd$ Now we are in position to control $\displaystyle \limsup_n n^{\frac{r}{d}} e_{r,n}(P,\phi_{_F})^r$. 
It follows from~\eqref{eq:comapardistir}, as  $|\widetilde \Gamma_n\cap U|\le |\widetilde \Gamma_n|\le n$, that
\begin{align*}
n^{\frac{r}{d}} e_{r,n}(P,\phi_{_F})^r &\le  n^{\frac{r}{d}} e_{r,n}(\widetilde \Gamma_n\cap U,P,\phi_{_F})^r +C_{d,r,L,\nabla^2F}\|h_m-h\|_1\ve^{-\frac{r}{d}} \\
      &=n^{\frac{r}{d}}\int\min_{a\in\widetilde \Gamma_n\cap U}\Bigg (\Big (\int_0^1(1-u)\nabla^2F(a+u(\xi-a))\mathrm{d}u\Big )(\xi-a)^{\otimes2}\Bigg)^{\frac{r}{2}}P_m(\mathrm{d}\xi)\\
& \quad +C_{d,r,L,\nabla^2F}\|h_m-h\|_1\ve^{-\frac{r}{d}}\\
&\leq n^{\frac{r}{d}}\int\min_{a\in\Gamma_n}\Bigg (\Big (\int_0^1(1-u)\nabla^2F(a+u(\xi-a))\mathrm{d}u\Big )(\xi-a)^{\otimes2}\Bigg)^{\frac{r}{2}}P_m(\mathrm{d}\xi)\\
& \quad +C_{d,r,L,\nabla^2F}\|h_m-h\|_1\ve^{-\frac{r}{d}}.
\end{align*}
This holds for any quantizer $\Gamma_n \subset U$ so that
\[
n^{\frac{r}{d}} e_{r,n}(P,\phi_{_F})^r \leq n^{\frac{r}{d}} e_{r,n_1(\ve)}(P_m,\phi_{_F})^r+C_{d,r,L,\nabla^2F}\|h_m-h\|_1\ve^{-\frac{r}{d}}.
\]
 }
Letting $n$ go to $+\infty$ yields
\begin{align}
  \nonumber   \limsup_{n\rightarrow+\infty} n^{\frac{r}{d}} e_{r,n}(P,\phi_{_F})^r & \leq (1-\ve)^{-\frac{r}{d}} 2^{-\frac r2}Q_r([0,1]^d) \big\|[\det(\nabla^2 F_m)]^{\frac{r}{2d}} h_m \big\|_{L^{\frac{d}{d+r}}(\lambda_d)} \\
    \label{eq:limsupPphir}&\qquad + C_{d,r,L,\nabla^2 F} \|h_m - h\|_1 \ve^{-\frac{r}{d}}.
\end{align}

$\rhd$ Now we deal with the $\displaystyle \liminf_n$ of the normalized $(r,\phi_{_F})$-distorsion. Let us consider again a quantizer $\Gamma_{n,\ve} \subset U$ with size $n_1(\ve)$. With the same notations as those used  in~Step~5, we derive from~\eqref{eq:comapardistir} that for every $n\ge\frac{1}{\ve}\vee \frac{1}{1-\ve}$, 
\begin{align*}
n^{\frac rd} e_{r}(\Gamma_{n,\ve},P,\phi_{_F})^r & \ge n^{\frac rd} e_r(\widetilde \Gamma_n \cap U, P_m, \phi_{_F})-C_{d,r,L,\nabla^2F} \ve^{-\frac rd} \|h_m-h\|_{L^1(\lambda_d)}\\
&\ge n^{\frac rd}  e_{r,n}(P_m, \phi_{_F})^r -C_{d,r,L,\nabla^2F} \ve^{-\frac rd} \|h_m-h\|_{L^1(\lambda_d)}
\end{align*}
since $\Gamma_{n,\ve} \subset \widetilde \Gamma_n\cap U$ and $ |\widetilde \Gamma_n\cap U|\le n$. As the right hand side of the above inequality no longer depends on $\Gamma_{n,\ve}$, this implies that
\[
n^{\frac rd} e_{r,n(\ve)} (P,\phi_{_F})^r \ge n^{\frac rd}  e_{r,n}(P_m, \phi_{_F})^r -C_{d,r,L,\nabla^2F} \ve^{-\frac rd} \|h_m-h\|_{L^1(\lambda_d)}.
\]
Now, as $n\mapsto n_1(\ve)$ is surjective from $\N$ onto $\N$, since $1-\ve\!\in (0,1)$ and non-decreasing to $+\infty$, 
\[
\lim_n n^{\frac rd}e_{r,n} (P,\phi_{_F})^r = \lim_n n_1(\ve)^{\frac rd}e_{r,n_1(\ve)} (P,\phi_{_F})^r.
\]
Consequently
\[
\lim_n n^{\frac rd}e_{r,n} (P,\phi_{_F})^r \cdot \lim_n \Big(\frac{n}{n_1(\ve)}\Big)^{\frac rd} \ge  \liminf_n n^{\frac rd}  e_{r,n}(P_m, \phi_{_F})^r -C_{d,r,L,\nabla^2F} \ve^{-\frac rd} \|h_m-h\|_{L^1(\lambda_d)}
\]
i.e.
\[
\lim_n n^{\frac rd}e_{r,n} (P,\phi_{_F})^r \ge (1-\ve)^{\frac rd}\Big( \liminf_n n^{\frac rd}  e_{r,n}(P_m, \phi_{_F})^r -C_{d,r,L,\nabla^2F} \ve^{-\frac rd} \|h_m-h\|_{L^1(\lambda_d)}\Big).
\]

Hence
\textcolor{black}{
\begin{align}
\nonumber \liminf_n n^{\frac rd}e_{r,n} (P,\phi_{_F})^r &\ge (1-\ve)^{\frac rd}\big( \lim_n n^{\frac rd}  e_{r,n}(P_m, \phi_{_F})^r - C_{d,r,L,\nabla^2F} \ve^{-\frac rd} \|h_m-h\|_{L^1(\lambda_d)}\big)\\
\nonumber&\ge (1-\ve)^{\frac rd} \big(2^{-\frac r2}Q_r([0,1]^d) \big\|[\det(\nabla^2 F_m)]^{\frac{r}{2d}} h_m \big\|_{L^{\frac{d}{d+r}}(\lambda_d)}\\
& \hskip 4cm  - C_{d,r,L,\nabla^2F} \ve^{-\frac rd} \|h_m-h\|_{L^1(\lambda_d)}\big),
\label{eq:liminfPphir}
\end{align}
}
where we used the conclusion of Step~3 in the second line.

\smallskip
\noindent  {\sc Step 6} {\em  (Differentiation of measure, $P$ absolutely continuous)}. \label{diffmeasure}

To conclude in the compactly supported case (for absolutely continuous distributions $P$), we need to let $m\to +\infty$ in the upper and lower  bounds established in Step~5. To this end, we have to prove that 

\begin{equation}\label{eq:hmtoh}
\|h_m-h\|_{L^1(\lambda_d)}\to 0\quad \mbox{ as }\quad m\to +\infty
\end{equation} 
and 
\begin{equation}\label{eq:detFhmtodetFh}
\big\|[\det(\nabla^2 F)_m-\tilde\ve_mI_d)^{\frac{r}{2d}}\cdot h_m\big \|_{L^{\frac{d}{d+r}}(\lambda_d)}\to \big\|[\det(\nabla^2 F)]^{\frac{r}{2d}} \cdot h \big\|_{L^{\frac{d}{d+r}}(\lambda_d)}\; \mbox{as}\; m\to +\infty 
\end{equation}
with $\tilde \ve_m= 2\ve _m$ or $0$.

Since  we assume in this step that $P$ is absolutely continuous then $h= \frac{dP^a}{\mathrm{d}\lambda_d}= \frac {dP}{\mathrm{d}\lambda_d}$ is a probability density function (w.r.t. the Lebesgue measure) so that $\|h\|_{L^1(\lambda_d)}=1$.
Then by Besicovitch's  differentiation of measure theorem (see e.g.~\cite[Chapter~VI]{Chatterji1973}):
$$
    h_m \rightarrow h \quad \lambda_d\mbox{-}a.s. \quad  \text{ as } \quad m\rightarrow +\infty.
$$
As $h$ an $h_m$ are both probability densities, hence non-negative   with an integral over $U$ equal to $1$, it follows from  Scheff\'e's Lemma that
\begin{equation}\label{hconv}
    \| h_m - h\|_{L^1(\lambda_d)} \rightarrow 0\quad  \text{ as } \quad m\rightarrow +\infty.
\end{equation} 

Before dealing  with the second convergence, we need to establish some $L^{\frac{d}{d+r}}(\lambda_d)$-convergence results and control on $h_m-h$ and $h_m$ respectively. Indeed, using H\"older's inequality with conjugate exponents $1+\frac rd$ and $1+ \frac dr$, we obtain
\begin{align}
 \label{eq:hmd/d+r} \big\|h_m\big\|^\frac{d}{d+r}_{L^{\frac{d}{d+r}}(\lambda_d)} & =\int_{\bar U_{\ve_0}}h^{\frac{d}{d+r}}_m\mathrm{d}\lambda_d \le \Bigg(\int h_m(\xi)\mathrm{d}\xi \Bigg)^{1+\frac rd} \lambda_d(\bar U_{\ve_0}\cap C)^{1+\frac dr} = \lambda_d(\bar U_{\ve_0}\cap C)^{1+\frac dr} \\
\nonumber  \mbox{ and }\hskip 2.5cm  &\\
\nonumber\big \|h-h_m\big\|^\frac{d}{d+r}_{L^{\frac{d}{d+r}}(\lambda_d)}&\le  \Bigg(\int (h-h_m)(\xi)\mathrm{d}\xi \Bigg)^{1+\frac rd} \lambda_d(\bar U_{\ve_0}\cap C)^{1+\frac dr}\\
\label{eq:h-hmd/d+r} & \le   \| h_m - h\|_{L^1(\lambda_d)}^{\frac{d}{d+r}}  \lambda_d(\bar U_{\ve_0}\cap C)^{1+\frac dr} \to 0 \quad \mbox{ as }\quad m\to +\infty. 
\end{align}

Now let us deal with the second quantity of interest. We focus on the case $\widetilde \ve _m =0$ since the case $\widetilde \ve_m = \ve_m$ can be handled likewise.  One checks from the very definition of $\nabla^2F_m$ and $\ve_m$ that, for every $\xi \!\in C \cap \bar U_{\ve_0}$,
\[
\nabla^2F_m(\xi) -  \nabla^2 F(\xi) =\left( \ve_mI_d + \sum_{i\in I_m} \textbf{1}_{C_i}\big( \nabla^2F(c_i)-\nabla^2F(\xi) \big)\right) 
\]
so that
\begin{align*}
\vertiii{\nabla^2F_m(\xi) - \nabla^2F(\xi) }& 
\le   \ve_m+ \sum_{i\in I_m} \textbf{1}_{C_i}
w\big(\nabla^2F,C_i,\tfrac{\sqrt{d}L}{m}\big) \\
& \le (\ve_m +\ve_m) = 2\ve_m. 
\end{align*}
This proves that
\[
\sup_{\xi \in C \cap \bar U_{\ve_0}} \vertiii{\textbf{1}_{\cup_{i\in I_m}C_i} \left(\nabla^2F_m(\xi)  -  \nabla^2 F(\xi)\right)}\le 2\ve_m \to 0\quad \mbox{ as}\quad m\to +\infty.
\]
One shows likewise that 
\[
\sup_{\xi  \in C \cap \bar U_{\ve_0}} \vertiii{\textbf{1}_{\cup_{i\in I_m}C_i}  \left(\nabla^2F_m(\xi) -\ve_mI_d -  \nabla^2 F(\xi)\right)} \le 3\, \ve_m \to 0\quad \mbox{ as}\quad m\to +\infty.
\]

Now note that  $\vertiii{\nabla^2F(\xi)} $ is bounded on the compact $ C \cap \bar U_{\ve_0}$ since $\nabla^2 F $ is continuous on this compact. The same holds true  for $\xi \!\in \cup_{i\in I_mC_i}$  since  $\nabla^2F_m(\xi)= \nabla^2F_m(c_i)$  with $c_i \!\in C\cap \bar U_{\ve}$ if $\xi\!\in C_i$. Consequently,  all these quantities lie in a compact subset $K_{F,\ve_0}$ of ${\cal S}^+(d,\R)\subset \mathbb{M}(d,\R)$, on which the continuous function $\det(\cdot)^{\frac{r}{2d}}$ is uniformly continuous and bounded on this compact. Consequently, 
 \begin{equation}\label{eq:cvgcedet}
\sup_{\xi \in  C \cap \bar U_{\ve_0}}
\bigg|\Big[\det \big(\nabla^2F_m(\xi))\big)\Big]^{\frac{r}{2d}} - \Big[\det \big( \tfrac 12 \nabla^2F(\xi)\big)\Big]^{\frac{r}{2d} }\bigg| \longrightarrow 0   \quad \mbox{ as}\quad m\to +\infty.
 \end{equation}

Now, applying the pseudo-triangle inequality 
\[
\|f+g\|_{L^s(\lambda_d)}^s\le \|f\|_{L^s(\lambda_d)}^s+\|g\|_{L^s(\lambda_d)}^s
\]
satisfied by the pseudo-norms $\|\cdot\|_{L^s(\lambda_d)}$ when $s\!\in (0,1)$, we get 
\begin{align*}
  \Bigg\| \Big [\det\! \big(\nabla^2F_m\big)\Big]^{\frac{r}{2d} }&\cdot h_m   - \Big[\det\! \big(  \nabla^2F\big)\Big]^{\frac{r}{2d} }\cdot h \Bigg\|^{\frac{d}{d+r}}_{L^{\frac{d}{d+r}}(\lambda_d)}\\
  &  \le \underbrace{ \sup_{\xi \in  C \cap \bar U_{\ve_0}}[ \det(\nabla^2F(\xi))]^{\frac{r}{2(d+r)}}}_{=:  C_{F,r,d,\ve_0} }\cdot \| h-h_m\|_{L^{\frac{d}{d+r}}(\lambda_d)}^{\frac{d}{d+r}}\\
\nonumber  &\qquad + \left\| 
 \left|\Big[\det\! \big(\nabla^2F_m(\xi))\big)\Big]^{\frac{r}{2d}}-\Big[\det\! \big( \nabla^2F(\xi)\big)\Big]^{\frac{r}{2d} }\right|h_m \right\|^{\frac{d}{d+r}}_{L^{\frac{d}{d+r}}(\lambda_d)}\\
\nonumber  & \le  C_{F,r,d,\ve_0} \| h-h_m\|_{L^{\frac{d}{d+r}}(\lambda_d)}^{\frac{d}{d+r}} \\
\nonumber  &\qquad+ \|h_m\|_{L^{\frac{d}{d+r}}(\lambda_d)}^{\frac{d}{d+r}}\!\!\left[\sup_{\xi \in C \cap \bar U_{\ve_0}}\!\! \left|\Big[\det \!\big(\nabla^2F_m(\xi))\big)\Big]^{\frac{r}{2d}}\!-\!\Big[\det\! \big( \nabla^2F(\xi)\big)\Big]^{\frac{r}{2d} }\right| \right]^{\frac{d}{d+r}}\\
\nonumber &\le   C_{F,r,d,\ve_0}\| h-h_m\|_{L^{\frac{d}{d+r}}(\lambda_d)}^{\frac{d}{d+r}}\\
 &\qquad + \left[\sup_{\xi \in C \cap \bar U_{\ve_0}} \left|\Big[\det\! \big(\nabla^2F_m(\xi))\big)\Big]^{\frac{r}{2d}}\!-\!\Big[\det\! \big( \nabla^2F(\xi)\big)\Big]^{\frac{r}{2d} }\right|  \right]^{\frac{d}{d+r}}\hskip-0.25cm \textcolor{black}{\lambda_d(\bar U_{\ve_0}\cap C)^{1+\frac dr} }
\end{align*}
where we used~\eqref{eq:hmd/d+r} in the last line to get rid of    $ \|h_m\|_{L^{\frac{d}{d+r}}(\lambda_d)}^{\frac{d}{d+r}}$. Then it follows from 
~\eqref{eq:cvgcedet} and~\eqref{eq:h-hmd/d+r}  that 
\begin{equation}\label{eq:detbregconv}
 \left\| \Big[\det\! \big(\nabla^2F_m\big)\Big]^{\frac{r}{2d} }\cdot h_m - \Big[\det \!\big(  \nabla^2F\big)\Big]^{\frac{r}{2d} }\cdot h \right\|^{\frac{d}{d+r}}_{L^{\frac{d}{d+r}}(\lambda_d)}\to 0\quad \mbox{ as}\quad m\to +\infty.
\end{equation}
At this stage, note  that for any sequence $g_m$, $m\ge 1$ of non-negative functions, if $s\!\in (0,1)$,  
\[
\Big|\int_{\R^d} g_m^s\mathrm{d}\lambda_d-\int_{\R^d} g^s\mathrm{d}\lambda_d\Big|\le \int_{\R^d} |g_m^s-g^s|\mathrm{d}\lambda_d\le \int_{\R^d}  |g_m-g|^s\mathrm{d}\lambda_d
\]
since the function $u\mapsto u^s$ is $s$-H\"older on the real line. Applying this elementary inequality to what precedes yields \textcolor{black}{the announced convergence~\eqref{eq:detFhmtodetFh}}.

\smallskip
\noindent $\rhd$ Now we can let $m\to +\infty$ in both~\eqref{eq:limsupPphir} and~\eqref{eq:liminfPphir} from Step~5. Using ~\eqref{hconv} and~\eqref{eq:detbregconv}, we obtain:
$$
    \limsup_{n\rightarrow+\infty} n^{\frac{r}{d}}e_{r,n}(P,\phi_{_F})^r\leq (1-\ve)^{-\frac{r}{d}}2^{-\frac r2}Q_r([0,1]^d)\|\det(\nabla^2F)^{\frac{r}{2d}}h\|_{L^{\frac{d}{d+r}}(\lambda_d)}
$$
and
\begin{equation*} 
\liminf_n n^{\frac rd} e_{r,n} (P,\phi_{_F})^r \ge (1-\ve)^{\frac rd} 2^{-\frac r2}Q_r([0,1]^d)\big\|[\det(\nabla^2F)]^{\frac{r}{2d}}h\big\|_{L^{\frac{d}{d+r}}(\lambda_d)}.
\end{equation*}

Finally, letting $\ve\rightarrow 0$ yields
\begin{equation}\label{eq:ZadorPbornee}
    \lim_{n\rightarrow+\infty}     n^{\frac{r}{d}} e_{r,n}(P,\phi_{_F})^r  = 2^{-\frac r2} Q_r([0,1]^d) \|\det(\nabla^2 F)^{\frac{r}{2d}}h\|_{L^{\frac{d}{d+r}}(\lambda_d)}.
\end{equation}

At this stage the theorem is proved for absolutely continuous distributions satisfying Assumption~\eqref{eq:HypoZadorBreg-chap2}$(i)$.

\smallskip
\noindent {\sc Step 7}
{\em  (The case of probability distributions with a singular component)}.

Rather than a direct proof adapted from that in~\cite[Section~6.2]{GrafL2000}, we will rely on the ``regular'' $L^r$-Zador's Theorem with the Euclidean norm as a loss function.  Assume $P=P^s$ i.e.  $P$ is singular  (still with a support contained in $\bar U_{\ve_0}$)  and let $\Gamma\subset U$ be a quantizer of size at most $n$. Let $[\nabla F]_{\rm Lip, \ve_0}$ denote the Lipschitz coefficient of $\nabla F$ on $C \cap \bar  U_{\ve_0}$, with in mind that $\nabla^2f $ is bounded on this convex set. One has for every quantizer $\Gamma\subset U$
\begin{align*}
e_r(\Gamma, P^s,\phi_{_F})^r & \le  \big(\tfrac 12 [\nabla F]_{\rm Lip, \ve_0}\big)^{\frac r2}  e_r(\Gamma, P^s, |\cdot|^2)^r
\end{align*}
owing to~\eqref{eq:ComparEuclide}.  Now, as the  closure $\bar U$ of $U$ in $\R^d$ is a nonempty closed convex and the projection ${\rm Proj}_{\bar U}:\R^d\to \bar U$ on $\bar U$  is $1$-Lipschitz and coincides with identity on $\bar U$,
\[
 e_r(\Gamma, P^s, |\cdot|^2)^r= \int_{U} \min_{a\in \Gamma}|\xi-a|^rP^s(\mathrm{d}\xi)\le \int_{U} \min_{a\in \Gamma}|\xi-{\rm Proj}_{\bar U} (a)|^rP^s(\mathrm{d}\xi).
\]
By the same argument, it is clear that 
\begin{align*}
e_{r,n}(P^s, |\cdot|^2) & = \inf\Bigg\{\int \min_{a\in \Gamma}|\xi-a|^rP^s(\mathrm{d}\xi) : \Gamma\subset \R^d, |\Gamma|\le n\Bigg\}\\
&=\inf\Bigg\{\int_{U} \min_{a\in \Gamma}|\xi- a|^rP^s(\mathrm{d}\xi):\Gamma\subset \bar U, \, |\Gamma|\le n \Bigg\}
\end{align*}
so that
\begin{align*}
e_{r,n}(P^s,\phi_{_F})^r &= \inf \Bigg\{ \int_{U} \min_{a\in \Gamma}\phi_{_F}(\xi,a)^{\frac r2}P^s(\mathrm{d}\xi): \Gamma\subset U,\, |\Gamma| \le n\Bigg\}\\
& \le \Big(\tfrac 12 [\nabla F]_{\rm Lip, \ve_0}\Big)^{\frac r2} e_{r,n}(P^s, |\cdot|^2)^r.
\end{align*}

Then, one concludes by regular $L^r$-Zador's Theorem for the (canonical) Euclidean norm (having in mind that  at this  stage $P= P^s$ is supposed to have a compact support included in $U$) that
\[
\limsup_n n^{\frac rd}e_{r,n}(P^s, \phi_{_F})^r\le   \big(\tfrac 12 [\nabla F]_{\rm Lip, \ve_0}\big)^{\frac r2}  \lim_n n^{\frac rd} e_{r,n}(P^s, |\cdot|^2)=0.
\]
Now let us deal with  the general case  $P= P^a +P^s$ where both measures are non zero. Let $\ve \!\in (0,1)$, let $n> \frac{1}{\ve}\vee \frac{1}{1-\ve}$  and let $n_1= n_1(\ve) = \lfloor(1-\ve)n\rfloor$ and $n_2(\ve) = \lfloor n\ve\rfloor$. One has $n_1+n_2 \le n$ so that if $\Gamma_n = \Gamma^{(1)}_{n_1} \cup\Gamma^{(2)}_{n_2}$ with $|\Gamma^{(i)}_{n_i}| \le n_i$, we derive 
\begin{align}
\label{eq:decompPaPs}e_r(\Gamma, P, \phi_{_F})^r & = P^a (U) e_r\Big(\Gamma, \frac{P^a}{P^a(U)}, \phi_{_F}\Big)^r+P^s(U) e_r\Big(\Gamma, \frac{P^s}{P^s(U)}, \phi_{_F}\Big)^r\\
\nonumber& \le P^a (U) e_r\Big(\Gamma^{(1)}_{n_1}, \frac{P^a}{P^a(U)}, \phi_{_F}\Big)^r+P^s(U) e_r\Big(\Gamma^{(2)}_{n_2}, \frac{P^s}{P^s(U)}, \phi_{_F}\Big)^r
\end{align}
so that 
\begin{align*}
e_{r,n}(P, \phi_{_F})^r & \le P^a (U) e_{r,n_1(\ve)}\Big( \frac{P^a}{P^a(U)}, \phi_{_F}\Big)^r+P^s(U) e_{r,n_2}\Big(\frac{P^s}{P^s(U)}, \phi_{_F}\Big)^r.
\end{align*}
\textcolor{black}{This in turn entails}
\begin{align*}
\limsup_n n^{\frac rd}e_{r,n}(\P, \phi_{_F})^r &\le P^a (U) \limsup_n \Big(\frac{n_1(\ve)}{n}\Big)^{\frac rd}n_1(\ve)^{-\frac rd} e_{r,n_1(\ve)}\Big(  \frac{P^a}{P^a(U)}, \phi_{_F}\Big)^r \\
&\qquad + P^s(U)  \limsup_n \Big(\frac{n_2(\ve)}{n}\Big)^{\frac rd} n_2(\ve)^{-\frac rd}     e_{r,n_(\ve)}\Big( \frac{P^s}{P^s(U)}, \phi_{_F}\Big)^r.
\end{align*}
Then, we get 
\begin{align*}
\lim_n n^{\frac rd}e_r(\Gamma, P, \phi_{_F})^r  & \le  P^a (U)(1-\ve)^{\frac rd} \lim_n n^{\frac rd}e_r\Bigg(\frac{P^a}{P^a(U)}, \phi_{_F}\Bigg)^r\\
&\qquad +  P^s (U)\ve^{\frac rd} \limsup_n n^{\frac rd}e_r\Big( \frac{P^s}{P^s(U)}, \phi_{_F}\Big)^r\\
&=   P^a (U)(1-\ve)^{\frac rd} \lim_n n^{\frac rd}e_r\Bigg(\frac{P^a}{P^a(U)}, \phi_{_F}\Bigg)^r +0.
\end{align*}
Letting $\ve\to 0$  yields using the result obtained in Step~6 for absolutely continuous $P$, 
\begin{align*}
\lim_n n^{\frac rd}e_r(\Gamma, P, \phi_{_F})^r & \le P^a (U) \lim_n n^{\frac rd} e_{r,n}\Bigg(\frac{P^a}{P^a(U)}, \phi_{_F}\Bigg)^r\\
&= \underbrace{ \int_U h\, \mathrm{d}\lambda_d}_{=1}\,2^{-\frac r2}Q_r([0,1]^d)\Big \|\det(\nabla^2 F)^{\frac{r}{2d}}  \frac{h}{\int_Uh\,\mathrm{d}\lambda_d}\Bigg\|_{L^{\frac{d}{d+r}}(\lambda_d)}\\
& = 2^{-\frac r2} Q_r([0,1]^d) \|\det(\nabla^2 F)^{\frac{r}{2d}} h\|_{L^{\frac{d}{d+r}}(\lambda_d)}.
\end{align*}
Starting again from~\eqref{eq:decompPaPs}, we  derive this time that 
\begin{align*}
e_{r}(\Gamma, P, \phi_{_F})^r & \ge P^a (U) e_{r,n_1(\ve)}\Bigg(\Gamma,  \frac{P^a}{P^a(U)}, \phi_{_F}\Bigg)^r 
\end{align*}
which yields by the same reasoning as above
\[
\liminf_n n^{\frac rd}e_r(\Gamma, P, \phi_{_F})^r   \ge    P^a (U)(1-\ve)^{\frac rd} \lim_n n^{\frac rd}e_r\Bigg(\frac{P^a}{P^a(U)}, \phi_{_F}\Bigg)^r 
\]
and, as a consequence,  by letting $\ve \to 0$
\[
\liminf_n n^{\frac rd}e_r(\Gamma, P, \phi_{_F})^r   \ge   2^{-\frac r2}  Q_r([0,1]^d) \|\det(\nabla^2 F)^{\frac{r}{2d}} h\|_{L^{\frac{d}{d+r}}(\lambda_d)}.
\]
At this stage the theorem is proved for distributions $P$ supported by $U$ and satisfying Assumption~\eqref{eq:HypoZadorBreg-chap2}$(i)$.

\smallskip
\noindent  {\sc Step 8} {\em  (Extension to the non-compact  case)}.
%
Let $K_k = [-k,k]^d\cap \bar U_{\frac 1k}^d$, $k\ge 1$, be a sequence of compact sets such that $\cup_{k\ge 1}^{\;\uparrow} K_k= U$ and let $P$ be a general distribution such that $P(U)=1$.

\smallskip
\noindent $\rhd$ {\em $\displaystyle \liminf_n$ side}. The  distribution $P$ can be decomposed into 
\begin{equation}\label{eq:decompP}
    P = P(K_k) P(\cdot\,|\,K_k)  + P(K_k^c) P(\cdot\,|\,K_k^c).
\end{equation}
In particular, 
\[
P \ge P(K_k) P(\cdot\,|\,K_k)  
\]
so that
\[
e_{r,n}(P, \phi_{_F})^r \ge  P(K_k) e_{r,n}(P_k, \phi_{_F})^r.
\]
As $P_k$ has a compact support included into $U$, what precedes implies that
\[
\lim_n n^{\frac rd}e_{r,n}(P_k, \phi_{_F})^r =  2^{-\frac r2}Q_r([0,1]^d) \Big \|[\det(\nabla^2F)]^{\frac{r}{2d}} \frac{h\textbf{1}_{K_k}}{P(K_k)} \Big\|_{L^{\frac{d}{d+r}}(\lambda_d)}
\]
so that, for every $k\ge 1$, 
\begin{align*}
\liminf_n n^{\frac rd}e_{r,n}(P, \phi_{_F})^r &\ge P(K_k)  2^{-\frac r2}Q_r([0,1]^d)
 \Big \|[\det(\nabla^2F)]^{\frac{r}{2d}}\frac{h\textbf{1}_{K_k}}{P(K_k)} \Big\|_{L^{\frac{d}{d+r}}(\lambda_d)}\\
 & = 2^{-\frac r2}Q_r([0,1]^d)
 \Big \|[\det(\nabla^2F)]^{\frac{r}{2d}} h\textbf{1}_{K_k}\Big\|_{L^{\frac{d}{d+r}}(\lambda_d)}.
\end{align*}
Letting $k$ go to infinity implies
\[
\liminf_n n^{\frac rd}e_{r,n}(P, \phi_{_F})^r \ge 2^{-\frac r2}Q_r([0,1]^d)
 \Big \|[\det(\nabla^2F)]^{\frac{r}{2d}} h \Big\|_{L^{\frac{d}{d+r}}(\lambda_d)}\!\in (0, +\infty],
\]
owing to Beppo Levi's  monotone convergence theorem. This proves claim~$(b)$ of the theorem since no moment assumption has been made so far on $P$ nor on the global boundedness of the operator norms $\vertiii{\nabla^2F(x)}$ over $U$.

\smallskip 
\noindent $\rhd$ {\em $\displaystyle \limsup_n$ side under assumption~\eqref{eq:HypoZadorBreg-chap2}$(ii)$}. First note that it follows from the integrability assumption~\eqref{eq:rintegZador}  that  there exists $\delta >0$ such that 
\begin{equation}\label{eq:rintegZador2}
\int_{\R^d} |\xi|^{r(1+\delta)}P(\mathrm{d}\xi)=\int_{U} |\xi|^{r(1+\delta)}P(\mathrm{d}\xi) < + \infty.
\end{equation}
 This will be the key for this step of the proof (and the only place where it will be called upon). As $\nabla F$ is Lipschitz on $U_{\eta}$ if $\eta>0$ or on $U$ if $\eta=0$ by assumption, we know that $F$ is sub-quadratic on these (convex) sets i.e.  
\[
|F(\xi)| \le C_{F,\eta}(1+|\xi|^2)
\]
so that under above moment assumption, the  $(r,\phi_{_F})$-Bregman quantization  have sense.
 
  Let $\ve \in (0,1)$. The  distribution $P$ can be decomposed into 
\begin{equation}\label{eq:decompP2}
    P = P(K_k) P(\cdot\,|\,K_k)  + P(K_k^c) P(\cdot\,|\,K_k^c).
\end{equation}
Set $n_1 =n_1 (\ve) =  \lfloor (1-\ve)n\rfloor$ and $n_2 =n_2 (\ve) = \lfloor \ve n\rfloor$ and let $\Gamma_{n_1}^{\ve, 1}$ and $\Gamma_{n_2}^{\ve,2}$ be quantizers  of size $n_1$ and  $n_2$ of the conditional distributions $P(\cdot\,|\,K_k)$ and $P(\cdot\,|\,K_k^c$) respectively such that 
\[
e_{r,n_1}\big(P(\,\cdot\,|\,K_k),\Gamma_{n_1}^{\ve,1},\phi_{_F}\big) \le e_{r,n_1}\big(P(\,\cdot\,|\,K_k),\phi_{_F}\big) \big(1+1/n_1\big)\]
and
\[
 e_{r,n_2}\big(P(\,\cdot\,|\,K^c_k),\Gamma_{n_2}^{\ve,2},\phi_{_F}\big) \le e_{r,n_2}\big(P(\,\cdot\,|\,K^c_k),\phi_{_F}\big) \big(1+1/n_2\big).
\]
Hence, 
\begin{align*}
    e_{r,n}(P,\phi_{_F})^r &\leq e_{r,n}(P,\Gamma_{n_1}^{\ve,1}\cup \Gamma_{n_2}^{\ve,2},\phi_{_F})^r\\
    &\leq P(K_k) e_{r,n_1}\big(P(\,\cdot\,|\,K_k),\Gamma_{n_1}^{\ve,1},\phi_{_F}\big)^r + P(K_k^c) e_{r,n_2}(P(\,\cdot\,|\,K_k^c),\Gamma_{n_2}^{\ve,2},\phi_{_F})^r\\
   & \leq  P(K_k) e_{r,n_1}\big(P(\,\cdot\,|\,K_k),\phi_{_F})^r (1+1/n_1\big)^r\\
   & \qquad+ P(K_k^c) e_{r,n_2}\big(P(\,\cdot\,|\,K_k^c),\phi_{_F}\big)^r (1+1/n_2)^r.
\end{align*}
Since $n_1+n_2\leq n$, then 
\begin{align}
    \nonumber\limsup_{n\rightarrow+\infty} n^{\frac{r}{d}} e_{n,r}&(P,\phi_{_F})^r \leq\\
    &  P(K_k) \lim_{n\rightarrow+\infty} 
   \Big (\frac{n}{\lfloor n(1-\ve)\rfloor}\Big )^{\frac{r}{d}}\limsup_{n\rightarrow+\infty} 
   \Big[n_1(\ve)e_{r,n_1}\big(P(\,\cdot\,|\,K_k),\Gamma_{n_1}^{\ve,1}\phi_{_F}\big)^r\Big] \\ 
    \nonumber &\qquad + P(K_k^c) \lim_{n\rightarrow+\infty} \Big ( \frac{n}{\lfloor\ve n\rfloor} \Big)^{\frac{r}{d}} \limsup_{n\rightarrow+\infty} 
    \Big[n_2(\ve)e_{r,n_2}(P(\,\cdot\,|\,K_k^c),\Gamma_{n_2}^{\ve,2},\phi_{_F})^r\Big] \\
    \nonumber &\le P(K_k) (1-\ve)^{-\frac{r}{d}} \limsup_{n\rightarrow+\infty} n_1^{\frac{r}{d}} e_{r,n_1}\big(P(\,\cdot\,|\,K_k),\phi_{_F}\big)^r \\
    \label{eq:MaojorCondit} & \qquad + P(K_k^c) \ve^{-\frac{r}{d}} \limsup_{n\rightarrow+\infty} n_2^{\frac{r}{d}} e_{r,n_2}\big(P(\,\cdot\,|\,K_k^c),\phi_{_F}\big)^r,
\end{align}
where we used in the last two lines that $1+1/n_i(\ve)\to 1$ for  $i=1,2$ as $n\to +\infty$. 

 We know from what precedes (Steps~4 and 5) that, as $K_k$ is a compact set, 
 \begin{align}
 \nonumber  \limsup_{n\rightarrow+\infty} n_1^{\frac{r}{d}} e_{r,n_1}\big(P(\,\cdot\,|\,K_k),\phi_{_F}\big)^r &=   \lim_{n\rightarrow+\infty} n_1^{\frac{r}{d}} e_{r,n_1}\big(P(\,\cdot\,|\,K_k),\phi_{_F}\big)^r \\
\label{eq:terme1}  &= 2^{-\frac r2} Q_r([0,1]^d)\Big\|\det(\nabla^2 F)^{\frac{r}{2d}} \frac{h\mathbf{1}_{K_k}}{P(K_k)}\Big\|_{L^{\frac{d}{d+r}}(\lambda_d)}.
 \end{align}

Now, to control the quantization error on (the non-compact set) $K_k^c$, we remark that $\nabla  F$ is Lipschitz continuous on $U_\eta$ if $\eta>0$ and on $U$ if $\eta =0$. It follows \textcolor{black}{from~\eqref{eq:ComparEuclide}}  \textcolor{magenta}{(see Property~\ref{property:2chap2})}
\begin{equation}\label{eq:majorregL^2}
\forall\, \xi, \,x \!\in U_{\eta}, \quad 0<\phi_{_F}(\xi,x) \le \tfrac 12 [\nabla F]_{U_{\eta}, \rm Lip} |\xi-x|^2.
\end{equation}
Using the same arguments as those used in Step~7 for singular distributions, we derive that for any distribution $Q$ such that $\int_U|\xi|^{r+\delta}Q(\mathrm{d}\xi)<+\infty$ for some $\delta>0$,  we deduce that 
\begin{equation}\label{leq:BregEuc}
    e_{r,n}(Q,\phi_{_F}) \leq\tfrac 12 [\nabla F]_{U_{\eta}, \rm Lip}  e_{r,n}(Q,|\cdot|^2),
\end{equation}
where  $[\nabla F]_{U_{\eta}, \rm Lip} <+\infty$ owing to Assumption~\eqref{eq:HypoZadorBreg-chap2}. Be careful that in regular optimal quantization theory, if we (temporarily) denote by  $e^{reg}_{r,n}(Q,|\cdot|)$ the $L^r$-optimal quantization error w.r.t the Euclidean norm, then $e_{r,n}(Q,|\cdot|^2)$ would read $e^{reg}_{r,n}(Q,|\cdot|)$.

This allows us to call upon $L^r$-Pierce's lemma in a non trivial way for such distributions  $Q$  applied here with respect to  the canonical Euclidean norm. By non trivial, we mean that the right hand side  of the below inequality is finite.  

\begin{lemma}{(Pierce Lemma for  regular quantization, see~\cite{LuPag23}[Corollary~2.1.13], \cite[Theorem~5.2$(b)$]{PagSpring2018}, \cite{LuPa2008} and \cite[Section 6.2]{GrafL2000})}\label{lem:Pierce} Let $d\ge1$. 
Let $r\geq1$ and let $\delta>  0$. There exists  a real constant $\widetilde C^{vor}_{d,r,\delta}>0$ such that for every distribution 
$Q$ on $\big(\mathbb{R}^d, {\cal B}or(\mathbb{R}^d)\big)$
\begin{equation*}
\forall\, n\ge 1,\quad     e^{reg}_{r,n}(Q,|\cdot |) \leq \widetilde C^{vor}_{d,r,\delta}\, n^{-\frac{1}{d}} \sigma_{r(1+\delta)}(Q),
\end{equation*}
where, for every $s>0$, $\displaystyle \sigma_s(Q) = \inf_{a\in\mathbb{R}^d}\Big(\int_{\mathbb{R}^d}|\xi-a|^sQ(\mathrm{d}\xi)\Big)^{1/s}\leq\Big(\int_{\mathbb{R}^d}|\xi|^sQ(\mathrm{d}\xi)\Big)^{1/s}$.
\end{lemma}

Let us apply this lemma with  $Q= P(\,\cdot\,|\,K_k^c)$ which satisfies the appropriate integrability assumption owing to the global integrability assumption~\eqref{eq:rintegZador2} satisfied by  $P$. This yields
\begin{align*}
\limsup_n n^{r/d}  e_{r,n}(P(\,\cdot\,|\,K^c_k),\phi_{_F})^r& \le \tilde C_{F,d,r,\delta} \left(\frac{\int_{K_k^c} |\xi|^{r+\delta} P(\mathrm{d}\xi)}{P(K_k^c)}\right)^{r/(r+\delta)}\\
&\le \tilde C_{F,d,r,\delta} \left(\frac{\int |\xi|^{r+\delta} P(\mathrm{d}\xi)}{P(K_k^c)}\right)^{r/(r+\delta)}
  \end{align*}
  where $C_{F,d,r,\delta}=  \big(\tfrac 12 [\nabla F]_{\rm Lip}C^{vor}_{d,r,\delta}\big)^r$.

Plugging~this inequality and~\eqref{eq:terme1} into~\eqref{eq:MaojorCondit} yields
\begin{align}
    \nonumber\limsup_{n\rightarrow+\infty} n^{\frac{r}{d}} e_{n,r}(P,\phi_{_F})^r &\leq  P(K_k) (1-\ve)^{-\frac{r}{d}} 2^{-\frac r2} Q_r([0,1]^d)\Big\|\det(\nabla^2 F)^{\frac{r}{2d}} \frac{h\mathbf{1}_{K_k}}{P(K_k)}\Big\|_{L^{\frac{d}{d+r}}(\lambda_d)} \\
  \nonumber  &\quad + P(K_k^c) \ve^{-\frac{r}{d}} C_{F,d,r,\delta} 
    \left(\frac{\int_{\mathbb{R}^d} |\xi|^{r+\delta} P(\mathrm{d}\xi) }{P(K_k^c)}  \right)^{r/(r+\delta)}\\
   \nonumber  & = (1-\ve)^{-\frac{r}{d}} 2^{-\frac r2} Q_r([0,1]^d) \Big\|\det(\nabla^2 F)^{\frac{r}{2d}} h\mathbf{1}_{K_k} \Big\|_{L^{\frac{d}{d+r}}(\lambda_d)}\\
   \nonumber   &\quad + P(K_k^c)^{\delta/(r+\delta)} \ve^{-\frac{r}{d}} C_{F,d,r,\delta} 
    \left(\int_{\mathbb{R}^d} |\xi|^{r+\delta} P(\mathrm{d}\xi)  \right)^{r/(r+\delta)}.
\end{align}

Now, note that  $ P(K_k^c)\to 0$ and $h\mathbf{1}_{K_k}\uparrow h$ as $k\rightarrow\infty$ so that by Beppo Levi's  monotone convergence  theorem (for the first term on the right hand side of the above inequality), for every $\ve \!\in (0,1)$, 
\begin{equation*}
   \limsup_{n\rightarrow+\infty} n^{\frac{r}{d}} e_{n,r}(P,\phi_{_F})^r \le (1-\ve)^{-\frac{r}{d}} 2^{-\frac r2}Q_r([0,1]^d) \Big\|\det(\nabla^2 F)^{\frac{r}{2d}} h \Big\|_{L^{\frac{d}{d+r}}(\lambda_d)}.
\end{equation*}

Then letting  $\ve\rightarrow 0$, yields
$$
\limsup_{n\rightarrow+\infty} n^{\frac{r}{d}}  e_{r,n}(P,\phi_{_F})^r \leq 2^{-\frac r2} Q_r([0,1]^d)\|\det(\nabla^2 F)^{\frac{r}{2d}}h\|_{L^{\frac{d}{d+r}}(\lambda_d)}.
$$
On the other hand, it follows from~\eqref{eq:decompP2} that 
\begin{equation}\label{eq:liminfineq}
    e_{r,n}(P,\phi_{_F})^r \geq P(K_k)e_{r,n}\big(P(\cdot\,|\,K_k),\phi_{_F}\big)^r
\end{equation}
which yields 
\begin{equation*}
    \liminf_{n\rightarrow +\infty}n^{\frac{r}{d}}e_{r,n}(P,\phi_{_F})^r \geq 2^{-\frac r2} Q_r([0,1]^d)\|\det(\nabla^2 F)^{\frac{r}{2d}}h\mathbf{1}_{C_k}\|_{L^{\frac{d}{d+r}}(\lambda_d)}.
\end{equation*}
Still, by the monotone convergence theorem, we obtain by  letting $k\rightarrow\infty$,
\begin{equation*}
    \liminf_{n\rightarrow +\infty}n^{\frac{r}{d}}e_{r,n}(P,\phi_{_F})^r \geq 2^{-\frac r2} Q_r([0,1]^d)\|\det(\nabla^2 F)^{\frac{r}{2d}}h\|_{L^{\frac{d}{d+r}}(\lambda_d)}.
\end{equation*}
This completes the proof of claim~$(a)$.


\medskip
\noindent $(b)$ The moment assumption is only involved in the last step (Step~8) to establish the $\displaystyle \limsup_n$ side of the sharp rate. As for the $\displaystyle \liminf_n$ side we only rely on~\eqref{eq:liminfineq} so that no moment assumption on $P$ is needed (beyond the one ensuring the existence of the $(r,\phi_{_F})$-mean quantization error). To be more precise we rely on the sharp rate for the case of distribution with compact support included in $U$ that we apply to the non-decreasing sequence $K_k$, $k$ large enough, introduced in Step~8. Then, for every $k$ large enough
\begin{align*}
\liminf_n n^{\frac 1d} e_{r,n}(P,\phi_{_F})&\ge P(K_k) \liminf_n  n^{\frac 1d} e_{r,n}\big(P(\cdot\,|\,K_k),\phi_{_F}\big)\\
							     & = P(K_k) \frac{1}{\sqrt{2}} Q_r([0,1]^d)^{\frac 1r}\Big\|{\rm det}(\nabla^2 F)^{\frac{r}{2d}}\frac{h}{P(K_k)}\mbox{\bf 1}_{K_k}\Big\|_{L^{\frac{d}{d+r}}(\lambda_d)}\\
							     &=\frac{1}{\sqrt{2}} Q_r([0,1]^d)^{\frac 1r}\Big\|{\rm det}(\nabla^2 F)^{\frac{r}{2d}}h\mbox{\bf 1}_{K_k}\Big\|_{L^{\frac{d}{d+r}}(\lambda_d)}.
\end{align*}
The conclusion follows by Beppo Levi's monotone convergence Theorem by letting $k\uparrow +\infty$ since $U= \bigcup^{\,\uparrow}_{k\ge 1} K_k$.

This completes the proof of Theorem~\ref{thm:ZadorBregman}. \hfill $\Box$

\bigskip
This result can be considered as slightly disappointing since it requires positive  definiteness  of $\nabla^2 F$, at least everywhere on $U$.

\begin{corollary}[An upper--bound when $F$ is simply $C^2$ and convex]. If $F:U\to \mathbb{R}^d$ is $C^2$, convex with a bounded Hessian on $U$. Then
\[
\limsup_n n^{\frac 1d} e_{r,n}(P, \phi_{_F}) \le \frac{1}{\sqrt{2}} Q_r([0,1]^d)^{\frac 1r} \Big\|{\rm det}(\nabla^2F)^{\frac{r}{2d}}h \Big\|^{\frac 1r}_{L^\frac{d}{d+r}(\lambda_d)}.
\]
\end{corollary}

\noindent {\em Proof.} For every $\ve\!\in (0,1)$, set 
\[
F_{\ve}(x) = F(x)+\ve |x|^2, \quad x\!\in\mathbb{R}^d.
\]

The function $F_{\ve}$ is strictly convex (in fact $\ve$-convex) with  $\nabla^2 F_{\ve} = \nabla^2F + 2\ve I_d$. Then, by linearity, 
\[
\phi_{\ve}(\xi,x)  = \phi_{_F}(\xi,x)   + \ve |\xi-x|^2
\]
Hence $\nabla^2 F_{\ve} = \nabla^2F + 2\ve I_d$.
Consequently $e_{r,n}(P, \phi_{_F}) \le e_{r,n}(P, \phi_{_{\ve}})$ so that
\begin{align}
\nonumber \limsup_n n^{r/d} e_{r,n}(P, \phi_{_F})^r& \le \limsup_n n^{r/d} e_{r,n}(P, \phi_{F_\ve})^r\\
\label{eq:upper-bound}
 &= Q_r([0,1]^d) \Big\|{\rm det}(\nabla^2F_{\ve})^{\frac{r}{2d}}h \Big\|_{L^{\frac{d}{d+r}}(\lambda_d)}.
\end{align}
Let us recall that Hadamard's inequality for determinants  reads on a square matrix $[a_{ij}]_{1\le i, j\le d}$
$$
\big|{\rm det}[a_{ij}]_{1\le i,j\le d} \big|\le \prod_{j=1}^d |[a_{\cdot j }]|\le \|A\|_{\rm Frob}^d,
$$
where $ \|A\|_{\rm Frob}= \sqrt{{\rm Tr}(AA^*)}$. As a consequence
\[
0\le {\rm det}\big(\nabla^2F(x)+2\ve I_d\big) \le C_d\big(\vertiii{ \nabla^2F}_{\sup}+ 2\sqrt{d}\ve\big)^d
\]
remains  bounded as $x$ varies. The moment $(r+\delta)$-assumption made on $P$ implies that $h\!\in L^{\frac {d}{d+r}}(\lambda_d)$ so that by dominated convergence,
\[
\Big\|{\rm det}(\nabla^2F_{\ve})^{\frac{r}{2d}}h \Big\|_{L^{\frac{d}{d+r}}(\lambda_d)}\to \Big\|{\rm det}(\nabla^2F)^{\frac{r}{2d}}h \Big\|_{L^{\frac{d}{d+r}}(\lambda_d)}\quad \mbox{as}\quad \ve \to 0.
\]
This completes the proof by letting $\ve \to 0$ in~\eqref{eq:upper-bound}.
\hfill $\Box$ 

\section{Zador's Theorem for matrix-valued fields of symmetric positive definite matrices}\label{sec:Zadorfieldmatrix}
Assume  that $F$ is twice differentiable. One checks that when $\xi$ and $a$ are close enough in $U$, then
\begin{equation*}
    \phi_{_F}(\xi, a) \simeq \tfrac 12(\xi - a)^T \nabla^2 F(a)(\xi - a).
\end{equation*}
As a consequence, at least when the quantization level $n$ is large,  and with the exception of the codewords $a$ which correspond to unbounded Bregman Voronoi cells
\[ 
\min _{a\in \Gamma} \phi_{_F}(\xi,a) \simeq  \tfrac 12\min_{a\in \Gamma} (\xi - a)^T \nabla^2 F(a)(\xi - a).
\]
Hence, one may reasonably guess  that using
\begin{equation}\label{eq:Hloss}
 H_{_F}(\xi,a)= (\xi - a)^T \nabla^2 F(a)(\xi - a)
\end{equation}
as a similarity  instead of $\phi_{_F}$ will produce a similar quantization (or classification), up to a $\sqrt{2}$ factor in terms of resulting error. 

Note by the way that, by Schwartz's Theorem, any such vector field is the Hessian of a $C^2$-strictly convex function.

This also suggests to directly study fields of symmetric positive  definite matrices. This is the object of the Theorem below whose proof is quite similar to that developed for the Bregman divergence $\phi_{_F}$.

\begin{theorem}[Zador like theorem for fields of positive definite matrices]\label{thm:ZadorFields}
Let  $U\subset \mathbb{R}^d$ be a nonempty open convex subset of $\mathbb{R}^d$, let $S : U\subset \mathbb{R}^d \rightarrow {\cal S}^{++}(d,\R)$ be a   continuous matrix valued vector field such that 
\[
 \forall\, x\!\in U, \quad S(x) \!\in{\cal S}^{++}(d,\mathbb{R}).
\]

\noindent $(a)$ Let  $P$ a probability distribution supported by $U$  i.e. $P(U)=1$ such that 
$$
\int_U |\xi|^{r+\delta}P(\mathrm{d}\xi)<+\infty\quad\mbox{for some $\delta>0$.}
$$ 
If
\begin{equation}\label{eq:HypoZadorField}
\left\{\begin{array}{ll}
(i) &{\rm supp}(P)\mbox{ is compact and included in $U$}\\
\mbox{or }\qquad &\\
(ii)& \exists \,\eta \ge 0 \; \mbox{ s.t. }\; {\rm supp}_U(P)\subset U_{\eta} \;\mbox{ and }\; \sup_{x\in U_\eta}\vertiii{S(x)}<+\infty
\end{array}\right.
\end{equation}
 with the convention  $U_0= U$, then (with obvious notations for the induced mean quuanization errors)
\begin{equation*}
    \lim_{n\rightarrow+\infty} n^{\frac1d} e_{r,n}(P,H_{_F}) =  Q_r([0,1]^d)^{\frac 1r} \big\|\det(S)^{\frac{r}{2d}}\cdot h \big\|^{\frac 1r}_{\frac{d}{d+r}},
\end{equation*}
where $h=\frac{\mathrm{d}P^a}{\mathrm{d}\lambda_d}$ denotes the density of the absolutely continuous part $P^a$ of $P$ with respect to the Lebesgue  measure $\lambda_d$ on $\mathbb{R}^d$. 

\smallskip
\noindent $(b)$ For any distribution $P$ supported by $U$, one has 
\begin{equation*}
    \liminf_{n\rightarrow+\infty} n^{\frac 1d } e_{r,n}(P,H_{_F}) \ge    Q_r([0,1]^d)^{\frac 1r}\big\|\det(S)^{\frac{r}{2d}}\cdot h\big\|^{\frac 1r}_{\frac{d}{d+r}}.
\end{equation*}
\end{theorem}

We do not detail the proof  which is quite similar to that for the Bregman divergence $\phi_{_F}$. In particular  it relies on the same firewall lemma to establish the lower bounds which is the most demanding part of the proof.

\bigskip
\noindent {\bf Remarks.}  $\bullet$ {\em Shrinking may help}. One can still use the ``shrinking may help'' trick by replacing $U$ by a   subset $V$ such that
\[
V\subset U, \;V \; \mbox{ is convex and} \;\sup_{x\in V}\vertiii{S(x)}<+\infty
\]
which allows, like  for  the existence of optimal quantizers theorems, to weaken the boundedness assumption  made on $S$. This is more general than the above assumption which corresponds to choose $V=U_{\eta}$ when $\eta>0$ (having in mind that $U_{\eta}$ is convex).

\smallskip
\noindent $\bullet$ This theorem  confirms the intuition  that the quantization w.r.t. the Bregman divergence and $\phi_{_F}$ and the Hessian field $S$  have the same sharp rate of quantization up to an obvious $\sqrt{2}$-factor. It  also enlightens the main difference between  quantizations based on Bregman divergence and powers of norms as a similarity measure. By construction quantization w.r.t.. the power of a norm is {\em isotropic} in the senses  that if $\Gamma$ is an optimal quantizer at a level $|\Gamma|$ for (the distribution of)  the random variable $X$ then its translation $a+\Gamma$ will be an optimal quantizer at the same level for (the distribution of)  the translated random variable $a+X$ for any $a\!\in \R^d$. This is clearly no longer the case for optimal quantization  based on $H_{_F}$ as a loss function~\eqref{eq:Hloss} and, due  to their proximity, for optimal quantization w.r.t.. the Bregman divergence $\phi_{_F}$. 

For this reason, it is not clear at all that the recent improvements of  Zador's Theorem obtained in~\cite{LuPag23} for ``regular'' quantization by  similarity measure which  is a power of a norm, can be extended  to our Bregman divergence framework.  To be more precise,  the fact that, when the distribution $P$ is radial, or almost radial in some sense outside a compact set, the moment assumption on the distribution $P$ can be optimally weakened: typically a finite $r$-moment for $P$  is enough when dealing with $L^r$-quantization based on the similarity function  of the form $N(\cdot)^r$, $N(\cdot)$ norm on $\R^d$ to get the sharp rate of Zador's Theorem. Trying to answer this open  question will be the object of future work.
\section*{Appendix: Proof of the firewall lemma}\label{sec:appendix}

\color{black}We first recall for the reader convenience the statement of this key lemma. We adopt the notations of Section~\ref{thm:ZadorBregman}.

\bigskip
\noindent  \bf Proposition} (Firewall Lemma)  {\it Let $C_i\subset C \cap\bar U_{\ve_0}$, $i\!\in I_m$,  be a closed hypercube with edges parallel to the coordinate  axis with length $L/m>0$ and  center $c_i$.
Let $\varpi \!\in (0, L/2m]$ and let $C_{i,\varpi} $ be the  hypercube with edge-length $L/m-2\varpi$ obtained as the image of $C_i$ by the contraction centered at $c_i$ with ratio $1-\varpi$  (see Figure~\ref{fig:firewall}).
Then there exists a finite set $\gamma_i = \gamma_i^{(\varpi)}\subset \partial C_{i,\varpi}$ (boundary of $C_{i,\varpi}$) such that 
$$
\forall \xi \in C_{i,\varpi}, \quad \min_{a\in\gamma_i} \phi_{_F}(\xi,a) \leq \min_{x\in C\setminus C_{i} } \phi_{_F}(\xi,x).
$$
Moreover the cardinality of the sets $\gamma_i$, $i\!\in I_m$,  only depends on the operator norm \linebreak $\displaystyle \vertiii{\nabla^2 F}_{C \cap \bar U_{\ve_0}} := \sup_{\xi\in C \cap \bar U_{\ve_0}}\vertiii{\nabla^2 F(\xi)}$, $\varpi$, $L$, $d$ and the uniform continuity modulus  $w(\nabla^2F, C \cap \bar U_{\ve_0},\cdot)$ on the compact  $C \cap \bar U_{\ve_0}$.}
\color{black}

\medskip 
\noindent \noindent {\em Proof.}
Let $[0,1]^d$. It is clear by a standard covering argument based on the $\ell^{\infty}$-norm  that for very $\rho\!\in (0,1)$, there exists a set $\gamma^{(\rho)}$ of points of $\partial [0,1]^d$ such that 
\[
\forall\, \xi \!\in \partial [0,1]^d, \; \exists\, a \!\in \gamma_{\rho} \quad \mbox{ such that} \quad |\xi-a |\le \rho.
\] 
Let us denote $\nu(\rho)= |\gamma^{(\rho)}|$ the cardinality of $\gamma^{(\rho)}$.

Let us consider for any  $i\!\in I_m$ the hypercube $C_i$ centered at $c_i$ with edges parallel to the coordinate axis and common edge length $\frac Lm-2\varpi$. The hypercube $C_i$ is the image of $[0,1]^d$ by the  similarity defined as the composition of a  translation by a vector $c_i-\frac 12\textbf{1}_d$ with  a dilatation centered at $c_i$ with ratio $\frac Lm$. Consequently the image $\gamma^{(\rho)}_{i,\varpi}$ of $\gamma^{(\rho)}$ by this translation-dilatation satisfies 
\[
\forall\, \xi \!\in \partial C_{i,\varpi}, \; \exists\, a \!\in \gamma^{(\rho)}_{i,\varpi} \quad\mbox{such that}\quad |\xi-a |\le \rho\Big(\frac{L}{m}-2\varpi\Big)\le \rho\frac Lm.
\] 
Throughout the rest of the proof we will denote $\gamma_i$ for simplicity instead of $\gamma^{(\rho)}_{i,\varpi}$.  All these sets $\gamma_i$ clearly have the same cardinality $\nu(\rho)$.

As $F$ is $C^2$, $\nabla^2F$ is uniformly continuous on $\bar U_{\ve_0} $,   there exists for every $\eta>0$ a $\rho= \rho(\eta)$ that we can always choose in $(0, \eta]$ such that the continuity modulus of $\nabla^2 F$ for the operator norm $\vertiii{\,\cdot\,}$ satisfies, 
\begin{equation}\label{eq:unifw}
\forall\, \delta\!\in (0, \rho],\quad w(\nabla^2F , C_i, \rho) \le w(\nabla^2 F, \bar U_{\ve_0},\delta)  \le \eta.
\end{equation}

This $\eta$ will be specified later independently of  $\rho$.

\smallskip
Let $x\!\in C \setminus C_i$ and let $\xi \!\in C_{i, \varpi}$. We consider a generic element $\zeta$ of the geometric segment $(\xi,x)$ of 
the form  
$$
\zeta =\zeta_{\lambda}:= \lambda x +(1-\lambda) \xi= \xi +\lambda (x-\xi),\;\lambda \!\in (0,1).
$$
We want to evaluate $\phi_{_F}(\xi,x)-\phi_{_F}(\xi,\zeta)$ for various values of $\zeta$ as a preliminary computation.

We start by the representation of the Bregman divergence based on  the Taylor formula with integral remainder
\begin{align}
\label{eq:identitphi_F0}\phi_{_F}(\xi,x) &=  \int_0^1 u\nabla^2F(\xi+u(x-\xi))(x-\xi)^{\otimes2}\mathrm{d}u  \\
\label{eq:identitphi_F1}		   & =  \int_0^\lambda u\nabla^2F(\xi+u(x-\xi))(x-\xi)^{\otimes2}\mathrm{d}u  + \int_{\lambda}^1 u\nabla^2F(\xi+u(x-\xi))(x-\xi)^{\otimes2}\mathrm{d}u  
\end{align}
The change of variable $u = \lambda v$ in the first integral yields
\begin{align*}
  \int_0^\lambda u\nabla^2F(\xi+u(x-\xi))(x-\xi)^{\otimes2}\mathrm{d}u & =\lambda^2 \int_0^1 v\nabla^2F(\xi +\lambda v (x-\xi))(x-\xi)^{\otimes 2}\mathrm{d}v\\
  & = \int_0^1 v\nabla^2F(\xi + v(\zeta-\xi))(\zeta-\xi)^{\otimes 2}\mathrm{d}v \\
  & = \phi_{_F}(\xi,\zeta).
\end{align*}
Consequently (with the change of variable $u=1-v$ in the second integral), it follows from~\eqref{eq:identitphi_F1}
\begin{equation}\label{eq:identitephiF}
\phi_{_F}(\xi,x) = \phi_{_F}(\xi,\zeta) + \int_0^{1-\lambda}  (1-v) \nabla^2F(x+v(\xi-x))(\xi-x)^{\otimes2}\mathrm{d}u.
\end{equation}
In particular 
\begin{equation}\label{eq:minorphiF}
\phi_{_F}(\xi,x)\ge \phi_{_F}(\xi, \zeta).
\end{equation}

Let
$$
\mbox{$\Theta(\xi,x) = \xi+ \lambda(\xi,x)(x-\xi)$ where $\lambda(\xi,x)$ is such that  $\Theta(\xi,x) \in \partial C_i$}
$$
 and 
$$
\mbox{$\tau(\xi,x) = \xi+ \hat \lambda(\xi,x)(x-\xi) \;$ where $\;\hat \lambda(\xi,x)$ is such that $\;\tau(\xi,x) \in \partial C_{i,\varpi}$}
$$
as depicted (twice) in Figure~\ref{fig:firewall}.
\begin{figure}[h!]
    \centering
    \includegraphics[scale=0.5]{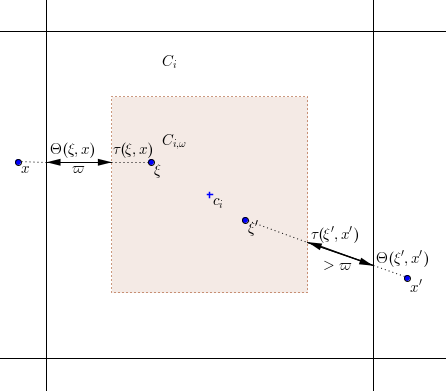}
    \caption{Firewall lemma.}
    \label{fig:firewall}
\end{figure}

It follows from~\eqref{eq:minorphiF} that 
\[
\phi_{_F}(\xi, x)\ge  \phi_{_F}(\xi,\Theta(\xi,x)).
\]
Hence 
\begin{equation}\label{eq:bordCisuffit}
    \inf_{x\in C\setminus \mathring{C}_i} \phi_{_F}(\xi,x) = \inf_{x\in \partial C_i} \phi_{_F}(\xi,x).
\end{equation}
    
Now, setting $\zeta = \tau(\xi,x)$ we derive from~\eqref{eq:identitephiF} that 

\begin{equation}\label{eq:minorphixix}
    \phi_{_F}(\xi,x) = \phi_{_F}(\xi,\tau(\xi,x)) + \int_{0}^{1-\hat\lambda(\xi,x)} (1-v) \nabla^2 F(x+u(\xi-x)))(\xi-x)^{\otimes2}\mathrm{d}u.
\end{equation}

 
For every $y \in \partial C_{i,\varpi}$, it follows from Step~1 that there exists $a_y \!\in \gamma_i$ such that  $|y-a_{y}|\leq \rho$. Then, for every $x\!\in C_i^c$  and for every $ \xi \!\in C_{i,\varpi} $ there exists $a_{\xi,x} = a_{\tau(\xi,x)}$ such that 
$$
|\tau(\xi,x)-a_{\xi,x}|\leq \rho\frac Lm.
$$ 
 
%

For $\xi \!\in C_{i, \varpi}$, let us note $\tilde a_\xi\! \in {\rm argmin}_{b\in \gamma_i} \phi_{_F}(\xi,b)$
a nearest Bregman neighbour of $\xi$ in $\gamma_i$.  

\smallskip
First we note that, by the definition of $a_\xi$, 
 $$
 \phi_{_F}(\xi,\tilde a_\xi)   \le \phi_{_F}(\xi, a_{\xi,x}).
 $$
 
 Starting from~\eqref{eq:identitphi_F0} applied with $\xi$ an $a_{\xi,x}$
\begin{align*}
    \nonumber \phi_{_F}(\xi,\tilde a_{\xi})  & \le \phi_{_F}(\xi, a_{\xi,x}) \\
       \nonumber  							& = \int_0^1 u\nabla^2F(\xi+u(a_{\xi,x}-\xi))(a_{\xi,x}-\xi)^{\otimes 2}\mathrm{d}u\\
   \nonumber   & = \int_0^1 u\,\big(\nabla^2F(a_{\xi,x}+u(x-\xi))- \nabla^2 F(\xi +u(\tau(\xi,x)-x))\big)(\xi-a_{\xi,x})^{\otimes2}\mathrm{d}u \\
   \nonumber   & \quad + \int_0^1 u\,\nabla^2 F(\xi +u(\tau(\xi,x)-x))(\xi-a_{\xi,x})^{\otimes 2}\mathrm{d}u\\
  			 &\le \int_0^1 u\, w(\nabla^2F, C_i, u|a_{\xi,x} -\tau (\xi,x)|)|\xi-a_{\xi,x}|^2\mathrm{d}u \\
\nonumber & \quad +  \int_0^1 u\,\nabla^2 F(\xi +u(\tau(\xi,x)-x))(\xi-a_{\xi,x})^{\otimes 2}\mathrm{d}u\\
 & \le \tfrac 12 w(\nabla^2F, C_i, |\ a_{\xi,x} -\tau (\xi,x)|)|\xi- a_{\xi,x}|^2 \\
         \nonumber   & \quad+  \int_0^1 u\,\nabla^2 F(\xi +u(\tau(\xi,x)-x))(\xi-a_{\xi,x})^{\otimes 2}\mathrm{d}u
    \end{align*}   
   so that 
   \begin{align}
     \nonumber \phi_{_F}(\xi,\tilde a_{\xi})   & \le \tfrac 12  w\big(\nabla^2F, \bar U_{\ve_0},\rho\tfrac Lm\big) \tfrac{dL^2}{m^2}\\
     \label{eq:majorphiFxiatildexi}     &\quad +  \underbrace{ \int_0^1 u\,\nabla^2 F(\xi +u(\tau(\xi,x)-x))(\xi-a_{\xi,x})^{\otimes 2}\mathrm{d}u}_{=:I},
 \end{align}   
where we used in the last line that $ |\ a_{\xi,x} -\tau (\xi,x)|\le \rho\frac Lm$, $|\xi- a_{\xi,x}|^2\le \sup_{y, z\in C_i }|y-z|^2 = \Big(\frac{dL^2}{m^2}\Big)$ and $ w(\nabla^2F, C_i, \rho\frac Lm)\le w\big(\nabla^2F, \bar U_{\ve_0},\rho\tfrac Lm\big)\le \eta$ by~\eqref{eq:unifw} since $m\ge L$.
    
\smallskip
Now we develop the integral $I$ in the right hand side of the last line. This yields
\begin{align}
   \nonumber  I& =  \underbrace{\int_0^1 u\,\nabla^2 F(\xi +u(\tau(\xi,x)-x))(\xi-\tau(\xi,x))^{\otimes 2}\mathrm{d}u}_{= \phi_{_F}(\xi, \tau(\xi,x))}\\
   \nonumber  &\quad +2 \int_0^1 u\,\nabla^2 F(\xi +u(\tau(\xi,x)-x))(\xi-\tau(\xi,x),\tau(\xi,x)-a_{\xi,x})\mathrm{d}u\\
   \nonumber  & \quad + \int_0^1 u\,\nabla^2 F(\xi +u(\tau(\xi,x)-x))(\tau(\xi,x)-a_{\xi,x}))^{\otimes 2}\mathrm{d}u\\
   \nonumber  &\le  \phi_{_F}(\xi, \tau(\xi,x)) +  \tfrac 12 \sup_{y \in C_i}\vertiii{\nabla^2 F(y)}\Big(2\,\sqrt{d}\tfrac Lm\cdot  \rho \tfrac Lm + \big(\rho \tfrac Lm\big)^2\Big)\\
\label{eq:majorI}  & \le   \phi_{_F}(\xi, \tau(\xi,x)) +  \tfrac 12 \big(\tfrac Lm\big)^2 \rho \sup_{y \in \bar U_{\ve_0}}\vertiii{\nabla^2 F(y)}\Big(2\,\sqrt{d} +  \rho \Big).
 \end{align}   

%
Combining~\eqref{eq:minorphixix},~\eqref{eq:majorphiFxiatildexi} and~\eqref{eq:majorI} yields,
\begin{align*}
    \phi_{_F}(\xi,x) - \phi_{_F}(\xi,\tilde a_\xi) &\geq \int_{0}^{1-\hat\lambda(\xi,x)} (1-u)\nabla^2 F(x+u(\xi-x))(\xi-x)^{\otimes2} \mathrm{d} u\\
    & \qquad -   \tfrac 12\Big( \big(\tfrac Lm\big)^2 \rho \sup_{y \in \bar U_{\ve_0}}\vertiii{\nabla^2F(y)}\Big(2\,\sqrt{d} +  \rho \Big)+ \eta \tfrac{dL^2}{m^2} \Big).
\end{align*}

Now, we use the $\nabla^2 F$ is uniformly elliptic on $\bar U_{\ve_0}$ with lower bound $c_0$ (see Equation~\eqref{eq:ellipticity} in Step~1 of the proof of the  theorem). Consequently
\begin{align*}
  \int_{0}^{1-\hat\lambda(\xi,x)} (1-u)\nabla^2 F(x+u(\xi-x))(\xi-x)^{\otimes2} \mathrm{d} u &\geq c_0 \int_{0}^{1-\hat\lambda(\xi,x)} (1-u)\mathrm{d}u  |x-\xi|^2   \\
    &= \frac{c_0}{2} (1-\hat\lambda(\xi,x)^2) |x-\xi|^2\\
    &\geq \frac{c_0}{2} (1-\hat\lambda(\xi,x)^2)\varpi^2 
\end{align*}
since $\displaystyle \inf_{x\in C\setminus \mathring{C}_i ,\xi\in C_{i, \varpi}} |x-\xi|= \varpi$.
As $x$, $\tau(\xi,x)$ and $\xi$ are aligned one has by  construction  for $\xi \!\in C_{i, \varpi} $ and $x\!\in C\setminus \mathring{C}_i$
\color{black}
\[
\hat\lambda(\xi,x)= \frac{|\tau(\xi,x)-\xi|}{|x-\xi|}  = \frac{|x-\xi|-|\tau(\xi,x)-x|}{|x-\xi|}\le  \frac{|x-\xi|-\varpi}{|x-\xi|}<1
\]
so that, taking advantage of~\eqref{eq:bordCisuffit}, 
\begin{align*}
\sup_{\xi\in C_{i, \varpi},x\in C\setminus \mathring{C}_i }\hat \lambda(\xi,x)&\le C_{\varpi, L, m}:=\sup_{\xi\in C_{i, \varpi},x\in (\partial{C}_i) }\hat \lambda(\xi,x)<1
\end{align*}
since  $C_{i, \varpi}\times  \partial{C}_i$ is compact and $(\xi,x)\mapsto |x-\xi|$ is continuous. Hence
\[
1-\hat \lambda(\xi,x)^2 \ge 1 -C_{\varpi, L, m}^2>0.
\]
Finally, using that $\rho $ is chosen in $(0, \eta]$,
\begin{align*}
    \phi_{_F}(\xi,x) - \phi_{_F}(\xi,\tilde a_\xi) &\geq \frac{c_0}{2}\varpi^2(1 -C_{\varpi, L, m}^2)  -   \frac \eta 2\big(\tfrac Lm\big)^2 \Big(  2\, \sup_{y \in \bar U_{\ve_0}}\vertiii{\nabla^2F(y)} \sqrt{d} + d \Big).
\end{align*}
We can fix  $\eta>0$ small enough so that the righthand side of the above inequality to be positive. Then let $\rho= \rho(\eta)$ satisfying~\eqref{eq:unifw}. Then we specify the sets $\gamma_i$ for each $C_i$ attached to this $\rho$. For such $\gamma_i$, we have for every $\xi\!\in C_{i,\varpi}$ and every $x\!\in C\setminus C_i$, 
\[
\min_{a\in \gamma_i} \phi_{_F}(\xi, a)= \phi_{_F}(\xi, \tilde a_{\xi}) \le \phi_{_F}(\xi,x).
\]
\color{black}
%
This completes the proof. 
\hfill $\Box$

\bibliographystyle{plain} 
\bibliography{these_artic.bib}

\end{document}